\documentclass[11pt]{amsart}

\usepackage{amssymb}
\usepackage{amsmath}
\usepackage{enumerate}
\usepackage[latin1]{inputenc}
\usepackage[dvips]{graphicx}

\def\s{\mathbb{S}}
\def\h{\mathbb{H}}
\def\r{\mathbb{R}}

\def\c{\mathbb{C}}
\def\p{\mathbb{P}}

\newtheorem{theorem}{Theorem}
\newtheorem{proposition}{Proposition}
\newtheorem{corollary}{Corollary}
\newtheorem{lemma}{Lemma}

  \theoremstyle{definition}
\newtheorem{definition}{Definition}
\newtheorem{example}{Example}

  \theoremstyle{remark}
\newtheorem{remark}{Remark}

\numberwithin{equation}{section}

\DeclareMathOperator{\sn}{sn}
\DeclareMathOperator{\cn}{cn}
\DeclareMathOperator{\dn}{dn}

\begin{document}

\title[Surfaces with parallel mean curvature vector]{Surfaces with parallel mean curvature vector in $\s^2\times\s^2$ and $\h^2\times\h^2$}

\author{Francisco Torralbo}
\address{Departamento de Geometr\'{\i}a  y Topolog\'{\i}a \\
Universidad de Granada \\
18071 Granada, SPAIN} \email{ftorralbo@ugr.es}

\author{Francisco Urbano}
\address{Departamento de Geometr\'{\i}a  y Topolog\'{\i}a \\
Universidad de Granada \\
18071 Granada, SPAIN} \email{furbano@ugr.es}

\thanks{Research partially supported by a MCyT-Feder research project MTM2007-61775 and a Junta Andalucía  Grant P06-FQM-01642.}

%\subjclass{Primary 53C42, 53B25; Secondary 53A05, 53D12}

%\keywords{Hamiltonian minimal Lagrangian surfaces, Willmore
%surfaces, Hamiltonian stationary Lagrangian spheres}

\date{}

\begin{abstract}
Two holomorphic Hopf differentials for surfaces of non-null parallel mean curvature vector in $\s^2\times\s^2$ and $\h^2\times\h^2$ are constructed. A 1:1 correspondence between these surfaces and pairs of constant mean curvature surfaces of $\s^2\times\r$ and $\h^2\times\r$ is established. Using that, surfaces with vanishing Hopf differentials (in particular spheres with parallel mean curvature vector) are classified and a rigidity result for constant mean curvature surfaces of $\s^2\times\r$ and $\h^2\times\r$ is proved.
\end{abstract}

\maketitle

\section{Introduction}
Surfaces with constant mean curvature (CMC-surfaces) in three manifolds is a classic topic in differential geometry and it has been extensively studied when the ambient manifold has constant curvature. In 2004, Abresh and Rosenberg \cite{AR} studied CMC-surfaces in $\s^2\times \r$ and $\h^2\times \r$ where $\s^2$ (respectively $\h^2$) are the two-dimensional sphere (respectively the hyperbolic plane). They defined on such surfaces a holomorphic two-differential which generalizes the classical Hopf differential defined for CMC-surfaces of space forms. They also classified those CMC-surfaces with vanishing Hopf-differential. In particular, they classified the orientable CMC-surfaces of genus zero in $\s^2\times \r$ and $\h^2\times\r$.

When the codimension of the surface is bigger than one, the natural generalization of these type of surfaces are the surfaces with parallel mean curvature vector (in what follows PMC-surfaces). Although there are results for codimension bigger than two, the most relevant ones are obtained when the codimension is two. In 1971, Ferus \cite{F} proved that a genus zero orientable surface with (non-null) parallel mean curvature vector in a simply-connected space form is a round sphere. In \cite{C} and \cite{Y}, Chen and Yau independently classified all the surfaces with parallel mean curvature vector in space forms, proving that they are  CMC-surfaces of three dimensional umbilical hypersurfaces. Both results are based on the following fact: if $H$ is the mean curvature vector of the surface, as the dimension of the normal bundle is two, it is possible to consider another parallel vector field in the normal bundle $\tilde{H}$ orthogonal to $H$ with the same length and to define two holomorphic Hopf differentials associated to $H$ and $\tilde{H}$.

In 2000, Kenmotsu and Zhou \cite{KZ} classified surfaces with parallel mean curvature vector in the complex  projective and the complex hyperbolic planes. In this case, it is well known that there are not umbilical hypersurfaces of these $4$-manifolds, therefore there is not a method like in space forms to construct surfaces of parallel mean curvature vector. The authors did not use the existence of Hopf differentials. Instead,  they reduced the classification theorem, using a result by Ogata \cite{O}, to solve an O.D.E. system on the surface. Using an analytic method, they classified these surfaces, proving that there are few of them and they have a good behavior with respect to the complex structure of the ambient space.

In this paper we study surfaces with parallel mean curvature vector in $\s^2\times\s^2$ and $\h^2\times\h^2$. In this case, although there are umbilical hypersurfaces of the ambient space, only the totally geodesic ones (up to congruences $\s^2\times\r$ and $\h^2\times\r$) have constant mean curvature (see Proposition~\ref{prop:subvariedades-totalmente-umbilicales-CMC}) and so CMC-surfaces of $\s^2\times\r$ and $\h^2\times\r$ are surfaces with parallel mean curvature vector in $\s^2\times\s^2$ and $\h^2\times\h^2$ respectively.

 The most important idea in the paper is the construction of two holomorphic Hopf differentials on  PMC-surfaces of $\s^2\times\s^2$ and $\h^2\times\h^2$ (see section 3) which generalize the Abresh-Rosenberg differential in the sense that if a PMC-surface of $\s^2\times\s^2$ or $\h^2\times\h^2$ factorizes through a CMC-surface of $\s^2\times\r$ or $\h^2\times\r$, both Hopf differentials are equals and they coincide (up to a constant) with the Abresh-Rosenberg differential (see Lemma~\ref{lm:diferencial-AR}). To define these Hopf differentials we use the two K\"ahler structures that these $4$-manifolds have (see section 3).

In section 4 we prove the main results of the paper. Theorem~\ref{tm:relacion-PMC-CMC} proves that given a simply-connected Riemannian surface $(\Sigma,g)$ there exists, up to congruences, a 1:1 correspondence between PMC-isometric immersions of $(\Sigma,g)$ in $\s^2\times \s^2$ (respectively $\h^2\times\h^2$) and pairs of CMC-isometric immersions of $(\Sigma,g)$ in $\s^2\times\r$ (respectively $\h^2\times\r$). Moreover these two CMC-surfaces are congruent if and only if the corresponding PMC-immersion factorizes through a CMC-immersion of $\s^2\times\r$ or $\h^2\times\r$. So the existence of full PMC-immersions in $\s^2\times \s^2$ and $\h^2\times\h^2$ is deeply related to the rigidity of CMC-immersions in $\s^2\times\r$ and $\h^2\times\r$.

In Theorem~\ref{tm:PMC-lagrangianas} we classify an important family of surfaces of $\s^2\times\s^2$ and $\h^2\times\h^2$ with parallel mean curvature vector: those that are Lagrangian surfaces with respect to some of the two K\"ahler structures that these manifolds have. Theorem~\ref{tm:PMC-K-normal-nula} is the most important contribution of the paper, it classifies the surfaces with parallel mean curvature vector with null extrinsic normal curvature. In the classification it appears the CMC-surfaces of $\s^2\times\r$ and $\h^2\times\r$, the Lagrangian PMC-surfaces and a new family of PMC-surfaces invariant under $1$-parameter groups of isometries of $\s^2\times \s^2$ and $\h^2\times\h^2$ which are described in Proposition~\ref{prop:familia-ejemplos-K-normal-0}. This result allows us to classify the parallel mean curvature surfaces with vanishing Hopf-differentials (Theorem~\ref{tm:PMC-diferenciales-nulas}) and in particular the parallel mean curvature spheres of $\s^2\times\s^2$ and $\h^2\times\h^2$ (Corollary~\ref{cor:PMC-spheres}).

In section 5, using Theorem~\ref{tm:relacion-PMC-CMC} and the examples of Proposition~\ref{prop:familia-ejemplos-K-normal-0}, we construct examples of CMC-surfaces in $\s^2\times\r$ and $\h^2\times\r$. Among them it is interesting to remark a two-parameter family of CMC-embedded tori in $\s^2\times\s^1$ (Proposition~\ref{prop:toros}). Moreover, Corollary~\ref{cor:rigidez} is a rigidity result for CMC-surfaces of $\s^2\times\r$ and $\h^2\times\r$. Finally, in section 6 we study general properties of compact PMC-immersions in $\s^2\times \s^2$ and $\h^2\times\h^2$.

 The product of two Riemannian surfaces with different constant curvatures is not an Einstein manifold and this is a big problem in order to study its PMC-surfaces. Following the ideas developed in this paper, on a PMC-surface of the product of two Riemannian surfaces with constant curvatures it is possible to define a holomorphic $2$-differential, which coincides with the sum of the two Hopf differentials when the constant curvatures are equals. When these curvatures are opposite, this holomorphic differential was defined in \cite{LV}.

\section{Preliminaries and examples}

We denote by $M^2(\epsilon)$, $\epsilon=1,-1$, the two-dimensional sphere $\s^2=\{x\in\r^3\,|\,x_1^2+x_2^2+x_3^2=1\}$ endowed with the canonical metric of constant curvature $1$ when $\epsilon=1$ and the hyperbolic plane  $\h^2=\{x\in\r^3\,|\,x_1^2+x_2^2-x_3^2=-1, \, x_3 > 0\}$ endowed with the canonical metric of constant curvature $-1$ when $\epsilon=-1$.  We denote by $\omega$ the K\"ahler $2$-form on $M^2(\epsilon)$ and by $J$ the corresponding complex structure, i.e. $\omega(\, , \, ) = \langle J\, , \rangle$ where $\langle\, , \, \rangle$ denotes the metric of $M^2(\epsilon)$.

If we consider $M^2(\epsilon)\times M^2(\epsilon)$ endowed with the product metric, which will be also denoted by $\langle,\rangle$, then it is an orientable Einstein manifold with scalar curvature $4\epsilon$. The orientation will be given by the $4$-form $\pi_1^*\omega\wedge\pi_2^*\omega$, where $\pi_j$, $j=1,2$ are the projections on the factors.

Along the paper we will consider $M^2(\epsilon)\times M^2(\epsilon)$ embedded isometrically in $\r^3\times\r^3\equiv\r^6$ when $\epsilon=1$ and in $\r^3_1\times\r^3_1\equiv\r^6_2$ when $\epsilon=-1$, being $\r^3_1$ the Lorentz-Minkowski $3$-space.

Let $\Phi:\Sigma\rightarrow M^2(\epsilon)\times M^2(\epsilon)$ be an immersion of an oriented surface $\Sigma$.
If $T^{\perp}\Sigma$ is the normal bundle of $\Phi$, then we have the orthogonal decomposition
\[
\Phi^*T(M^2(\epsilon)\times M^2(\epsilon))=T\Sigma\oplus T^{\perp}\Sigma.
\]
Let $\bar {\nabla}$ be the connection on $\Phi^*T(M^2(\epsilon)\times M^2(\epsilon))$ induced by the Levi--Civita connection of $M^2(\epsilon)\times M^2(\epsilon)$ and let $\bar{\nabla}=\nabla +\nabla^{\perp}$ be the corresponding decomposition.
If $\{e_1,e_2,e_3,e_4\}$ is an oriented orthonormal local frame on $\Phi^*T(M^2(\epsilon)\times M^2(\epsilon))$ such that $\{e_1,e_2\}$ is an oriented frame on $T\Sigma$, then we define the {\em normal curvature} $K^{\perp}$ of the immersion $\Phi$ by
\[
K^{\perp}=R^{\perp}(e_1,e_2,e_3,e_4),
\]
where $R^{\perp}$ is the curvature tensor of the normal connection $\nabla^{\perp}$.
Also we define the {\it extrinsic normal curvature} $\bar K^{\perp}$ as the the function on $\Sigma$ given by
\[
\bar K^{\perp}=\bar R(e_1,e_2,e_3,e_4),
\]
where $\bar{R}$ is the curvature tensor of $\bar\nabla$.
\begin{definition}
Let $\Phi:\Sigma\rightarrow\ M^2(\epsilon)\times M^2(\epsilon)$ be an immersion. We say that $\Phi$ has non-null parallel mean curvature vector, from now on PMC-immersion, if $\nabla^{\perp}H=0$ and $H$ is non-null. In such case, $|H|$ is a positive constant.
\end{definition}
Suppose that $\Phi:\Sigma\rightarrow M^2(\epsilon)\times M^2(\epsilon)$ is a PMC-immersion of an orientable surface $\Sigma$.
We can define another parallel normal vector field $\tilde{H}$ as the only one with $|\tilde{H}|=|H|$ and $\{\tilde{H},H\}$ defining the same orientation on the normal bundle
 as $\{e_3,e_4\}$. Because $H$ is parallel, $K^{\perp}=0$ and hence the Ricci equation of $\Phi$ is given by
\[
|H|^2\bar{K}^{\perp}=\langle\,[A_H,A_{\tilde{H}}]e_1,e_2\rangle,
\]
where $A_\xi$ is the Weingarten endomorphism associated to a normal vector $\xi$.

In order to get examples of PMC-surfaces, we make use of the following trivial fact: {\it If $\Sigma$ is a constant mean curvature surface of a totally umbilical hypersurface with constant mean curvature of $M^2(\epsilon)\times M^2(\epsilon)$, then $\Sigma$ has parallel mean curvature vector as a surface of $M^2(\epsilon)\times M^2(\epsilon)$.} Next proposition describes the umbilical hypersurfaces with constant mean curvature of $M^2(\epsilon)\times M^2(\epsilon)$.
\begin{proposition}\label{prop:subvariedades-totalmente-umbilicales-CMC}
Let $\Psi:N\rightarrow M^2(\epsilon)\times M^2(\epsilon)$ be a totally umbilical hypersurface with constant mean curvature. Then $\Psi$ is totally geodesic and it is locally congruent to the totally geodesic immersion:
\begin{align*}
\epsilon &= 1 & \epsilon &= -1 \\
\s^2\times\r &\rightarrow \s^2\times\s^2 & \h^2\times\r &\rightarrow \h^2\times\h^2 \\
(p,t) &\mapsto (p,(\cos t,\sin t,0)) & (p,t) &\mapsto (p,(0,\sinh t,\cosh t)).
\end{align*}
\end{proposition}

\begin{proof}
Let $\eta$ be a unit normal vector field of $N$ in $M^2(\epsilon)\times M^2(\epsilon)$, $\hat{\sigma}$ the second fundamental form of $\Psi$ and $\hat{H}$ the mean curvature. As $\Psi$ is totally umbilical we have that
\[
\hat{\sigma}(v,w)=\hat{H}\langle v,w\rangle\eta,\quad\forall\, v,w\,\in TN.
\]
As $\hat{H}$ is constant, we obtain that
\[
(\nabla\hat{\sigma})(x,v,w)=0,\quad \forall\, x,v,w\,\in TN.
\]
Now the Codazzi equation of $\Psi$ says that
\[
\bar{R}(x,v,w,\eta)=0,\quad \forall\, x,v,w\,\in TN.
\]
Using this property for suitable vectors and that $\hbox{dim}\,N=3$, it is easy to get that one of the two components of $\eta$ vanishes, and so, up to an isometry of $M^2(\epsilon)\times M^2(\epsilon)$,  we can take $\eta=(0,\eta_2)$.

If $\Psi=(\Psi_1,\Psi_2)$, then $\Psi$ and $\hat{\Psi}=(\Psi_1,-\Psi_2)$ are an orthogonal reference in the normal bundle of $M^2(\epsilon)\times M^2(\epsilon)$ in $\r^6$ or $\r^6_2$. So for any $v\in TN$, taking into account that $\langle \hat{\Psi}_*(v),\eta\rangle=-\langle v,\eta\rangle=0$, we have that
\[
v\cdot\eta=-\hat{A}_{\eta}v=-\hat{H}v.
\]
So the map $\eta+\hat{H}\Psi:N\rightarrow\r^6$ is a constant $A=(A_1,A_2)\in\r^3\times\r^3\equiv\r^6$, and hence $\hat{H}\Psi_1=A_1$ and $\eta_2+\hat{H}\Psi_2=A_2$. As $N$ is a $3$-manifold and $\Psi$ an immersion, $\Psi_1$ cannot be a constant and so $\hat{H}=0$ which implies that $\Psi$ is totally geodesic. Now the second equation says that $\eta_2=A_2$ and so $\langle\Psi_2,A_2\rangle=\langle\Psi,\eta\rangle=0$ with $|A_2|=|\eta_2|=1$. This proves that $\Psi_2(N)$ is a geodesic of $\s^2$ or $\h^2$ and the proof finishes.
\end{proof}

  As a consequence of this result we obtain that
\begin{quote}
  {\it CMC-surfaces of $M^2(\epsilon)\times\r$ are surfaces of $M^2(\epsilon)\times M^2(\epsilon)$ with parallel mean curvature vector}.
\end{quote}
Other examples of PMC-surfaces of $M^2(\epsilon)\times M^2(\epsilon)$ can be constructed in the following way: given two regular curves $\alpha:I\rightarrow M^2(\epsilon)$ and $\beta:I'\rightarrow M^2(\epsilon)$ then
\[
\begin{split}
\Phi:I\times I'&\rightarrow M^2(\epsilon)\times M^2(\epsilon) \\
\Phi(t,s)&=(\alpha(t),\beta(s))
\end{split}
 \]
is an immersion of the surface $I\times I'$ whose mean curvature vector is given by
\[
H=\frac{k_\alpha}{2}(J\alpha',0)+\frac{k_\beta}{2}(0,J\dot{\beta}),
\]
where $'$ (respectively $\dot{\,}$\,) stands for the derivative with respect to $t$ (respectively $s$), $k_{\alpha}$ and $k_{\beta}$ are respectively the curvatures of $\alpha$ and $\beta$ and we have assumed that $|\alpha'|=|\dot{\beta}|=1$. So we obtain that {\it $\Phi$ has parallel mean curvature vector if and only if $\alpha$ and $\beta$ are curves of constant curvature}. In that case, $4|H|^2=k_{\alpha}^2+k_{\beta}^2$, and hence $\Phi$ is minimal if and only if $\alpha$ and $\beta$ are geodesics. It is interesting to remark that the induced metric on $I\times I'$ by $\Phi$ is flat.

 Taking into account the curves of constant curvature of $\s^2$ and $\h^2$ we have that
the above examples are, up to congruences, open subsets of the following family of complete and embedded PMC-surfaces:
\begin{example}\label{ex:producto-curvas}
When $\epsilon=1$, the {\it tori product of two geodesic circles}
 \[
 T_{a,\hat{a}}=\{(x,y)\in\s^2\times\s^2\,/\,x_3=a,\,y_3=\hat{a}\},\quad 0\leq a\leq \hat{a}<1,\,a^2+\hat{a}^2>0,
 \]
 whose mean curvature satisfy $4|H|^2=\frac{a^2}{1-a^2}+\frac{\hat{a}^2}{1-\hat{a}^2}.$

When $\epsilon=-1$, we obtain three topological families of examples
\begin{enumerate}[1.]
\item the {\it tori product of two geodesic circles}
\[
\hat{T}_{a,\hat{a}}=\{(x,y)\in\h^2\times\h^2\,/\,x_3=a,\,y_3=\hat{a}\},\quad 1<a\leq \hat{a},
\]
whose mean curvature satisfy $4|H|^2=\frac{a^2}{a^2-1}+\frac{\hat{a}^2}{\hat{a}^2-1}$ and $|H|^2>1/2$,

\item the {\it cylinders product of a geodesic circle and a hypercycle}
\[
C_{a,b}=\{(x,y)\in\h^2\times\h^2\,/\,x_3=a,\,y_1=b\},\quad b\geq 0,\,a>1,
\]
whose mean curvature satisfy $4|H|^2=\frac{a^2}{a^2-1}+\frac{b^2}{b^2+1}$ and $|H|^2>1/4$,

and the {\it cylinders product of a geodesic circle and a horocycle}
\[
\hat{C}_{a}=\{(x,y)\in\h^2\times\h^2\,/\,x_3=a,\,y_1-y_3=1\},\quad a>1,
\]
whose mean curvature satisfy $4|H|^2=\frac{2a^2-1}{a^2-1}$ and $|H|^2>1/2$.

\item and finally the {\it planes product of two hypercycles}
\[
P_{b,\hat{b}}=\{(x,y)\in\h^2\times\h^2\,/\,x_1=b,\,y_1=\hat{b}\},\quad b,\hat{b}\geq 0, \, b\hat{b} \neq 0
\]
whose mean curvature satisfy $4|H|^2=\frac{b^2}{b^2+1}+\frac{\hat{b}^2}{\hat{b}^2+1}$ and $|H|^2<1/2$,

the {\it planes product of a hypercycle and a horocycle}
\[
\hat{P}_{b}=\{(x,y)\in\h^2\times\h^2\,|\,x_1=b,\,y_1-y_3=1\},\quad b\geq 0,
\]
whose mean curvature satisfy $4|H|^2=\frac{2b^2+1}{b^2+1}$ and $1/4\leq|H|^2<1/2$,

 and the {\it plane product of two horocycles}
 \[
 \tilde{P}=\{(x,y)\in\h^2\times\h^2\,|\,x_1-x_3=1,\,y_1-y_3=1\},
 \]
 whose mean curvature satisfies $|H|^2=1/2$.

\end{enumerate}
\end{example}

\section{Hopf differentials.}
In order to have a deep understanding of the geometry of $M^2(\epsilon)\times M^2(\epsilon)$ and of
its surfaces we need to introduce the two K\"ahler structures that $M^2(\epsilon)\times M^2(\epsilon)$ has. We can define two complex structures on $M^2(\epsilon)\times M^2(\epsilon) $ by
\[
J_1=(J,J),\quad J_2=(J,-J),
\]
whose K\"ahler two-forms are $\omega_1=\pi_1^*\omega+\pi_2^*\omega$ and $\omega_2=\pi_1^*\omega-\pi_2^*\omega$. Hence
\[
\omega_1\wedge\omega_1=-\omega_2\wedge\omega_2=2(\pi_1^*\omega\wedge\pi_2^*\omega),
\]
and so $J_1$ defines the chosen orientation on $M^2(\epsilon)\times M^2(\epsilon)$ and $J_2$ the opposite one.

 Now, $(M^2(\epsilon)\times M^2(\epsilon),\langle,\rangle,J_j)$, $j=1,2$ are K\"ahler-Einstein manifolds. It is clear that if $\mathrm{Id}:M^2(\epsilon)\rightarrow M^2(\epsilon)$ is the identity map and $F:M^2(\epsilon)\rightarrow M^2(\epsilon)$ is an anti-holomorphic isometry, then
\[
(\mathrm{Id},F):M^2(\epsilon)\times M^2(\epsilon)\rightarrow M^2(\epsilon)\times M^2(\epsilon)
\]
is a holomorphic isometry from $(M^2(\epsilon)\times M^2(\epsilon),\langle,\rangle,J_1)$ onto $(M^2(\epsilon)\times M^2(\epsilon),\langle,\rangle,J_2)$.

If $\Phi=(\phi,\psi):\Sigma\rightarrow M^2(\epsilon)\times M^2(\epsilon)$ is a PMC-immersion of an orientable surface $\Sigma$, then the {\it K\"ahler functions} on $\Sigma$, $C_1,C_2:\Sigma\rightarrow\r$, associated to the complex structures $J_1$ and $J_2$ are defined by
\[
\Phi^*\omega_j=C_j\omega_{\Sigma},\quad j=1,2,
\]
where $\omega_{\Sigma}$ is the area $2$-form of $\Sigma$. It is clear that $C_j^2\leq1$ and that the points where $C_j^2=1$ are the complex points of $\Phi$ with respect to the $J_j$ complex structure. It is interesting to remark that $C_j^2$ is well defined even when the surface is not orientable.

Now it is easy to check that the Jacobians of $\phi$ and $\psi$ are given by
\[
\hbox{Jac}\,(\phi)=\frac{C_1+C_2}{2},\quad\hbox{Jac}\,(\psi)=\frac{C_1-C_2}{2},
\]
and that the {\it extrinsic curvature} $\bar{K}=\bar{R}(e_1,e_2,e_2,e_1)$, where $\{e_1,e_2\}$ is an orthonormal frame on $T\Sigma$, and the normal extrinsic curvature are given by
\[
\bar{K}=\epsilon\frac{C_1^2+C_2^2}{2},\quad\quad \bar{K}^{\perp}=\epsilon\frac{C_1^2-C_2^2}{2}.
\]

We consider a local isothermal parameter $z=x+iy$ on $\Sigma$, such that
\[
\langle\Phi_z,\Phi_z\rangle=\langle\phi_z,\phi_z\rangle+\langle\psi_z,\psi_z\rangle=0,\quad
|\Phi_z|^2=|\phi_z|^2+|\psi_z|^2=e^{2u}/2,
\]
where the derivatives with respect to $z$ and $\bar{z}$ are given by
\[
\partial_z=\frac{1}{2}(\frac{\partial}{\partial x}-i\frac{\partial}{\partial y}),\quad \partial_{\bar{z}}=\frac{1}{2}(\frac{\partial}{\partial x}+i\frac{\partial}{\partial y}).
\]
We define two Hopf differentials in the following way:
\begin{eqnarray*}
\Theta_1(z)=\left(2\langle \sigma(\partial_z,\partial_z),H+i\tilde{H}\rangle+\frac{\epsilon}{4|H|^2}\langle J_1\Phi_z,H+i\tilde{H}\rangle^2\right)(dz)^2\\
\Theta_2(z)=\left(2\langle \sigma(\partial_z,\partial_z),H-i\tilde{H}\rangle+\frac{\epsilon}{4|H|^2}\langle J_2\Phi_z,H-i\tilde{H}\rangle^2\right)(dz)^2,
\end{eqnarray*}
where $\sigma$ is the second fundamental form of $\Phi$.

To prove that these Hopf differentials are holomorphic when the surface has parallel mean curvature vector, we need to study the Frenet equations of our immersion $\Phi:\Sigma\rightarrow M^2(\epsilon)\times M^2(\epsilon)\subset\r^6(\hbox{or}\, \r^6_2)$.

With the chosen orientation, $\{\Phi_x,\Phi_y,\tilde{H},H\}$ is an oriented frame on $\Phi^*T(M^2(\epsilon)\times M^2(\epsilon))$.
Denoting
\[
\xi=\frac{1}{\sqrt{2}|H|}(H-i\tilde{H}),
\]
we have that $|\xi|^2=1, \langle\xi,\xi\rangle=0$, $\nabla^{\perp}\xi=0$ and  $\{\xi,\bar{\xi}\}$ is a reference of the complexified normal bundle. Using known arguments in theory of surfaces in K\"ahler surfaces (see for instance~\cite{EGT}) and taking into account the chosen orientations it is easy to prove that
\begin{eqnarray}
J_1\Phi_z=iC_1\Phi_z+\gamma_1\xi,\quad\quad J_1\xi=-2e^{-2u}\bar{\gamma}_1\Phi_z-iC_1\xi,\label{eq:J1Phi-xi}\\
J_2\Phi_z=iC_2\Phi_z+\gamma_2\bar{\xi},\quad\quad J_2\xi=-2e^{-2u}\gamma_2\Phi_{\bar{z}}+iC_2\xi, \label{eq:J2Phi-xi}
\end{eqnarray}
for certain complex functions $\gamma_j,\,j=1,2$ which satisfy $|\gamma_j|^2=\frac{e^{2u}(1-C_j^2)}{2}$.

As $\{\Phi,\hat{\Phi}:=(\phi,-\psi)\}$ is an orthogonal reference along $\Phi$ of the normal bundle of $M^2(\epsilon)\times M^2(\epsilon)$ in $\r^6$ when $\epsilon=1$ and in $\r^6_2$ when $\epsilon=-1$, we easily get that the Frenet equations of the PMC-immersion $\Phi$ are given by
\begin{align*}
\Phi_{zz}&=2u_z\Phi_z+f_1\xi+f_2\bar{\xi}-\epsilon\frac{\gamma_1\gamma_2}{2}\hat{\Phi},\\
\Phi_{z\bar{z}}&=\frac{e^{2u}}{2}H-\epsilon\frac{e^{2u}}{4}\Phi-\epsilon\frac{e^{2u}}{4}C_1C_2\hat{\Phi},\\
\xi_z&=-\frac{|H|}{\sqrt{2}}\Phi_z-2e^{-2u}f_2\Phi_{\bar{z}}+\epsilon\frac{iC_1\gamma_2}{2}\hat{\Phi},\\
\bar{\xi}_z&=-\frac{|H|}{\sqrt{2}}\Phi_z-2e^{-2u}f_1\Phi_{\bar{z}}+\epsilon\frac{iC_2\gamma_1}{2}\hat{\Phi},
\end{align*}
for certain complex functions $f_j,\,j=1,2$.

Now, after a long and straightforward computation and taking into account the definitions of $C_j$ and $\gamma_j$, $j=1,2$ given in \eqref{eq:J1Phi-xi} and \eqref{eq:J2Phi-xi}, the integrability conditions of the above Frenet equations ($\Phi_{zz\bar{z}} = \Phi_{z\bar{z}z}$ and $\xi_{z \bar{z}} = \xi_{\bar{z}z}$) are given by
\begin{equation}\label{eq:Gauss-Ricci}
\left\{
\begin{array}{lr}
4u_{z\bar{z}}+e^{2u}(|H|^2+\epsilon\frac{C_1^2+C_2^2}{2})-4e^{-2u}(|f_1|^2+|f_2|^2)=0,
&\hbox{Gauss} \\
\epsilon e^{4u}(C_1^2-C_2^2) - 8(|f_1|^2-|f_2|^2) = 0,&\hbox{Ricci} \\
(\gamma_j)_z=2u_z\gamma_j-2iC_jf_j,& j=1,2,
\end{array}
\right.
\end{equation}
\begin{equation}
\left\{
\begin{aligned}
(C_j)_z&=2ie^{-2u}f_j\bar{\gamma}_j-i\frac{|H|}{\sqrt{2}}\gamma_j,&
(f_j)_{\bar{z}}&=i\epsilon\frac{e^{2u}C_j\gamma_j}{4},\\
(\gamma_j)_{\bar{z}}&=-\frac{i|H|C_je^{2u}}{\sqrt{2}},&
|\gamma_j|^2&=\frac{e^{2u}(1-C_j^2)}{2},
\end{aligned}
\right. \, j=1,2.
\label{eq:compatibilidad-PMC}
\end{equation}
 But the third equation in~\eqref{eq:Gauss-Ricci} can be easily deduced from the equations in~\eqref{eq:compatibilidad-PMC} and from that equation and using again~\eqref{eq:compatibilidad-PMC} we obtain Gauss and Ricci equations. So, really,  {\it the integrability conditions of the above Frenet system are~\eqref{eq:compatibilidad-PMC}}.

\begin{proposition}\label{prop:diferenciales-holomorfas}
Let $\Phi:\Sigma\rightarrow M^2(\epsilon)\times M^2(\epsilon)$ be a PMC-immersion of an orientable surface $\Sigma$. Then $\Theta_j$, $j=1,2,$ are holomorphic.
\end{proposition}

\begin{proof}
Using the functions defined above, the Hopf differentials $\Theta_j,\,j=1,2,$ can be written as
\[
\Theta_j=\left(2\sqrt{2}|H|f_j+\frac{\epsilon}{2}\gamma_j^2\right)\,(dz)^2,\quad j=1,2.
\]
Now, from the integrability conditions we obtain that $\left(4\sqrt{2}|H|f_j+\epsilon\gamma_j^2\right)_{\bar z}=0$, which proves the Proposition.
\end{proof}

From \eqref{eq:J1Phi-xi} and \eqref{eq:J2Phi-xi} we have that $\langle J_1\Phi_z,\bar{\xi}\rangle=\langle J_2\Phi_z,\xi\rangle=0$, and then the Hopf differentials can also be written as
\begin{eqnarray*}
\Theta_1(z)=\left(2\langle \sigma(\partial_z,\partial_z),H+i\tilde{H}\rangle-\frac{\epsilon}{|H|^2}\langle J_1\Phi_z,\tilde{H}\rangle^2\right)(dz)^2,\\
\Theta_2(z)=\left(2\langle \sigma(\partial_z,\partial_z),H-i\tilde{H}\rangle-\frac{\epsilon}{|H|^2}\langle J_2\Phi_z,\tilde{H}\rangle^2\right)(dz)^2.
\end{eqnarray*}

In the following result we compute these Hopf-differentials in the examples described in section 2.
\begin{lemma}~ \label{lm:diferencial-AR}
\begin{enumerate}[1.]
  \item Let $\Phi:\Sigma\rightarrow M^2(\epsilon)\times\r\hookrightarrow M^2(\epsilon)\times M^2(\epsilon)$ be a CMC-immersion. Then $\Theta_1=\Theta_2=2\Theta_{AR}$, where $\Theta_{AR}$ is the Abresh-Rosenberg holomorphic differential associated to $\Phi$ (see \cite{AR}).
  \item Let $\Phi:I\times I'\rightarrow M^2(\epsilon)\times M^2(\epsilon)$ be the product of two curves $\Phi(t,s)=(\alpha(t),\beta(s))$ of constant curvatures $k_{\alpha}$ and $k_{\beta}$ respectively. Then
    \[
    \Theta_j=
    \frac{\epsilon+4|H|^2}{16|H|^2}(k_{\alpha}+(-1)^ji\,k_{\beta})^2\,(dz)^2,\quad j=1,2.
    \]
\end{enumerate}
\end{lemma}

\begin{proof}
First we prove (1). It is clear that, in this case, $\eta=(0,(0,0,1))$ (respectively $\eta=(0,(1,0,0))$) when $\epsilon=1$ (respectively $\epsilon=-1$) is a unit normal field to the totally geodesic immersion $M^2(\epsilon)\times\r\hookrightarrow M^2(\epsilon)\times M^2(\epsilon)$ given in Proposition~\ref{prop:subvariedades-totalmente-umbilicales-CMC}. So $\tilde{H}=|H|\eta$. If $\tilde{\sigma}$ is the second fundamental form of $\Sigma$ in $M^2(\epsilon)\times\r$, then $\tilde{\sigma}=\sigma$ and then
\[
\langle \sigma(\partial_z,\partial_z),H+i\tilde{H}\rangle=\langle \sigma(\partial_z,\partial_z),H-i\tilde{H}\rangle=\langle \tilde{\sigma}(\partial_z,\partial_z),H\rangle.
\]
Also, if $\Phi=(\phi,\eta)$, from a direct computation we have that $\langle  J_1\Phi_z,\tilde{H}\rangle=|H|\eta_z$ and $\langle  J_2\Phi_z,\tilde{H}\rangle=|H|\eta_z$.

Finally, from the second expressions of $\Theta_j$, we get that
\[
\Theta_1=\Theta_2=(2\langle\tilde{\sigma}(\partial_z,\partial_z),H\rangle-\epsilon(\eta_z)^2)\,(dz)^2=2\Theta_{AR}.
\]
We remark that, in this case, the functions appearing in the Frenet equations are given by $f_1=f_2$, $\gamma_1=\gamma_2$ and so $C_1=C_2$.

The proof of (2) is trivial.
\end{proof}

From the integrability conditions we can get some properties and formulae about PMC-surfaces which will be used in the next sections.

\begin{enumerate}[$\bullet$]
    \item First, from the Gauss and Ricci equations joint with $4u_{z\bar{z}}=-Ke^{2u}$ it is easy to deduce that
        \begin{equation}
        |f_j|^2=\frac{e^{4u}}{8}(|H|^2-K+\epsilon C_j^2),\quad j=1,2.
        \label{eq:modulo-fj}
        \end{equation}
        These equations say that
         $$K\leq|H|^2+1,\quad \hbox{when}\,\,\epsilon=1$$
        and the equality is attained in a point $p$ if and only if for some $j\in\{1,2\}$ $f_j(p)=0$ and $C_j^2(p)=1$. Also
         $$K\leq|H|^2\quad\hbox{when}\,\,\epsilon=-1$$
        and the equality is attained in a point $p$ if and only if for some $j\in\{1,2\}$ $f_j(p)=0$ and $C_j(p)=0$.

    \item Second, using \eqref{eq:modulo-fj} and the integrability conditions, we obtain the following relation between $|\Theta_j|^2$ and $|\nabla C_j|^2$
        \begin{equation}
        \begin{split}
        &|\nabla C_j|^2+4\epsilon e^{-4u}|\Theta_j|^2 = \\
        &=(1-C_j^2+4\epsilon|H|^2)\left(\frac{\epsilon(1-C_j^2)}{4}+|H|^2+\epsilon C_j^2-K\right), \quad j =1,2.
        \end{split}
        \label{eq:gradiente-C-general}
        \end{equation}
    \item Also, from the integrability conditions, it is easy to compute the Laplacian of the K\"ahler functions $C_j$, obtaining
        \begin{equation}
        \Delta C_j=-C_j\left(4|H|^2-2K+\epsilon(1+C_j^2)\right),\quad j=1,2.
        \label{eq:laplaciano-C}
        \end{equation}
        This means that $C_j$ satisfies the equation $(\Delta+F)C_j=0$ where $F= 4|H|^2-2K+\epsilon(1+C_j^2)$. Then, using the Nodal Courant Theorem, we have that either $C_j=0$ or the set $\{p\in\Sigma\,|\,C_j(p)=0\}$ is a union of curves. In particular its interior is empty.

    \item To finish this section, it is interesting to remark that $\{p\in\Sigma\,/\/ C_j^2(p)=1\},\,i=1,2$ has empty interior, because if not, its interior is a non-empty complex surface and so it is minimal, contradicting that $|H|$ is a positive constant on $\Sigma$. But, under certain restrictions on the curvature of the surface, we can get more properties of these sets.
\end{enumerate}

\begin{proposition}\label{prop:C-tiene-puntos-aislados}
Let $\Phi:\Sigma\rightarrow M^2(\epsilon)\times M^2(\epsilon)$ be a PMC-immersion. If  $K(p)\not=\epsilon$, for any $p\in\Sigma$, then $\{p\in\Sigma\,|\,C_j^2(p)=1\}$, $j=1,2$ are sets of isolated points.

\end{proposition}

\begin{proof}
As the points $p$ with $C_j^2(p)=1$ are critical points of the function $C_j$, we are going to study the degeneracy of these points.

  Let $p_0$ be a point with $C_j(p_0)^2=1$, with $j\in\{1,2\}$. Then $\gamma_j(p_0)=0$ and from the integrability conditions~\eqref{eq:compatibilidad-PMC} one gets that
\begin{align*}
(C_j)_{zz}(p_0)&=-2\sqrt{2}|H|f(p_0)C(p_0),\\ (C_j)_{z\bar{z}}(p_0)&=-\frac{C_j(p_0)e^{2u(p_0)}|H|^2}{2}-4C_j(p_0)e^{-2u(p_0)}|f_j|^2(p_0).
\end{align*}
If $\Theta_j(p_0)=0$, then $f_j(p_0)=0$ and so the determinant of the Hessian of $C_j$ at $p_0$ is given by $|H|^4e^{4u(p_0)}$ which is not zero. Therefore, in this case, $p_0$ is a non-degenerate critical point.

If $\Theta_j(p_0)\not=0$, in a neighborhood of $p_0$ we can normalize $\Theta_j=\lambda\in\c^*$. In particular $\lambda=\Theta_j(p_0)=2\sqrt{2}|H|f_j(p_0)$. In this case we get
that the determinant of the Hessian of $C_j$ at $p_0$ is
\[
4\left(\frac{|\lambda|^2e^{-2u(p_0)}}{|H|^2}-|H|^2e^{2u(p_0)}\right)^2.
\]
So the critical point $p_0$ is degenerate if and only if $|\lambda|^2=|H|^4e^{4u(p_0)}$. But from \eqref{eq:modulo-fj} we obtain
\[
|\lambda|^2=|H|^2e^{4u(p_0)}(|H|^2+\epsilon- K(p_0)),
\]
and hence $p_0$ is degenerate if and only if $K(p_0)=\epsilon$. As the non-degenerate critical points are isolated, we finish the proof.
\end{proof}

\section{Main Results}
The integrability equations given in the previous section allow to relate, at least in the simply connected case, PMC-immersions in $M^2(\epsilon)\times M^2(\epsilon)$ with pairs of CMC-immersions in $M^2(\epsilon) \times \r$ with the same induced metric and the same length of the mean curvature. We concrete this relation in the following result.

\begin{theorem}\label{tm:relacion-PMC-CMC}
Given a simply-connected Riemannian surface $(\Sigma, g)$, there exists a 1:1 correspondence $[\Phi] \leftrightarrow ([\Phi_1], [\Phi_2])$, between congruent classes of PMC-isometric immersions $\Phi:(\Sigma, g)\rightarrow M^2(\epsilon) \times M^2(\epsilon)$ and pairs of congruent classes of CMC-isometric immersions $\Phi_1,\Phi_2:(\Sigma, g) \rightarrow M^2(\epsilon) \times \r$ with $|H| = |H_1| = |H_2|$, where $H$ is the mean curvature vector of $\Phi$ and $H_j$, $j = 1, 2$ are respectively the mean curvatures of $\Phi_j$, $j = 1, 2$. The Abresh-Rosenberg differentials $\Theta^j_{AR}$ associated to the pair of CMC-immersions $\Phi_j$, $j = 1, 2$, and the two Hopf differentials $\Theta_j$, $j = 1, 2$ associated to the PMC-immersion $\Phi$ are related by $2\Theta^j_{AR} = \Theta_j$, $j = 1, 2$.

Moreover, $[\Phi_1]=[\Phi_2]$ if and only if $\Phi$ factorizes
\[
\Phi:\Sigma\rightarrow M^2(\epsilon)\times\r\hookrightarrow M^2(\epsilon)\times M^2(\epsilon)
\]
through a CMC immersion in $M^2(\epsilon) \times \r$.
\end{theorem}

\begin{proof}
In order to proof this result we are going to use the integrability equations~\eqref{eq:compatibilidad-PMC} for PMC-conformal immersions given in the previous section  and the corresponding ones for CMC-immersions in $M^2(\epsilon) \times \r$ given in \cite{B}. As we work with conformal immersions we are going to use the conformal version of these equations obtained in~\cite{FM}, which can be described as follows.

Let $\Psi = (\psi, \eta): (\Sigma, g) \rightarrow M^2(\epsilon) \times \r$ be a CMC-isometric immersion with mean curvature $H$ and $z = x + iy$ a local isothermal parameter such that $g = e^{2u}|\mathrm{d} z|^2$. Then the Frenet equations of $\Psi: \Sigma \longrightarrow M^2(\epsilon) \times \r \subset \r^3\times \r$ (or $\r^3_1$) are given by:
\begin{equation}
\begin{aligned}
\Psi_{zz} &= 2u_z \Psi_z + p N + \epsilon \eta_z^2 \hat{\Psi} \\
\Psi_{z\bar{z}} &= \frac{e^{2u}}{2}H N + \epsilon \left( |\eta_z|^2 - \frac{e^{2u}}{2}\right) \hat{\Psi} \\
N_z &= -H\Psi_z - p \Psi_{\bar{z}} + \epsilon \eta_z \nu \hat{\Psi}
\end{aligned}
\label{eq:frenet-CMC}
\end{equation}
where $\hat{\Psi} = (\psi, 0)$, $N$ is a unit normal vector to the immersion $\Psi$, $p$ is a complex function and $\nu$ is a real function defined by $\nu = \langle N, (0,1)\rangle$.  The integrability equations of this Frenet system are given by (see~\cite[Theorem 2.3]{FM} for more details):
\begin{equation}
\begin{aligned}
  p_{\bar{z}} &= \epsilon \frac{e^{2u}}{2} \nu \eta_z, &  \nu_z &= -H \eta_z - 2e^{-2u} p \eta_{\bar{z}} \\
  \eta_{z\bar{z}} &= \frac{e^{2u}}{2}H \nu,\quad  &
  |\eta_z|^2 &= \frac{e^{2u}}{4}(1 - \nu^2)
\end{aligned}
\label{eq:compatibilidad-CMC}
\end{equation}

Now we prove the result. Let $\Phi: (\Sigma, g) \rightarrow M^2(\epsilon) \times M^2(\epsilon)$ be a PMC-isometric immersion of a simply-connected oriented surface $\Sigma$ and $z$ an isothermal parameter such that $g = e^{2u}|\mathrm{d} z|^2$. Using~\eqref{eq:compatibilidad-PMC} it is followed that $i(\gamma_j)_{\bar{z}}$ is a real function and so, because $\Sigma$ is simply-connected, there exists a function $\eta_j$ such that $i\gamma_j = \sqrt{2}(\eta_j)_z$, $j = 1, 2$. We consider the data
\[
\bigl(u, H_j = |H|, \nu_j = C_j, \eta_j, p_j = \sqrt{2}f_j\bigr), \quad j= 1,2.
\]
From~\eqref{eq:compatibilidad-PMC}, it is followed that these data satisfy~\eqref{eq:compatibilidad-CMC}, and so there exist two CMC-isometric immersions $\Phi_j: (\Sigma, g) \rightarrow M^2(\epsilon) \times \r$ with $|H_j| = |H|$, $j = 1, 2$.

Moreover, it is easy to check that if $\Phi$ is congruent to $\Psi$, then the co\-rres\-pon\-ding $\Phi_j$ and $\Psi_j$ are also congruent for $j=1,2$.

Conversely, let $\Phi_j = (\phi_j, \eta_j):(\Sigma, g) \rightarrow M^2(\epsilon) \times \r$ be two CMC-isometric immersions with $|H_1| = |H_2|$ and $z$ an isothermal parameter with $g = e^{2u}|\mathrm{d} z|^2$. We may suppose, composing with an appropriate isometry if necessary, that $H_1 = H_2 > 0$. We consider data
\[
  \left(u, |H| = H_1 = H_2, C_j = \nu_j, \gamma_j = -i\sqrt{2}(\eta_j)_z, f_j = \frac{p_j}{\sqrt{2}}:\, j = 1,2\right).
\]
From~\eqref{eq:compatibilidad-CMC}, it is followed that  these data satisfy~\eqref{eq:compatibilidad-PMC}, and so there exists a PMC-isometric immersion $\Phi: (\Sigma, g) \rightarrow M^2(\epsilon) \times M^2(\epsilon)$ with $|H| = |H_1| = |H_2|$.

Moreover, it is easy to check that if $\Phi_j$ are congruent to $\Psi_j$, $j=1,2$, then the corresponding $\Phi$ and $\Psi$ are also congruent.

Secondly, as the Abresh-Rosenberg differential for CMC-surfaces can be expressed as $\Theta^j_{AR} = (|H_j|p_j - \frac{\epsilon}{2} (\eta_j)_z^2)(\mathrm{d} z)^2$ and the Hopf differentials for PMC-surfaces as $\Theta_j=\left(2\sqrt{2}|H|f_j+\frac{\epsilon}{2}\gamma_j^2\right)\,(dz)^2$, using the above relations between the data, we obtain that $2\Theta^j_{AR}=\Theta_j$, $j=1,2$.

Finally, if $[\Phi_1]=[\Phi_2]$, then $\Phi_1, \Phi_2: (\Sigma, g) \rightarrow M^2(\epsilon)\times \r$ are two CMC-isometric immersions satisfiyng $\Phi_2 = F \circ \Phi_1$, where $F$ is an isometry of $M^2(\epsilon)\times \r$. Then, given an isothermal parameter $z$, and possible up to a congruence, we can take the data of $\Phi_j$ as $|H_1| = |H_2|$, $p_1 = p_2$, $\nu_1 = \nu_2$ and $\eta_1 = \eta_2$. Therefore the associated PMC-isometric immersion $\Phi = (\phi, \psi):(\Sigma, g) \rightarrow M^2(\epsilon)\times M^2(\epsilon)$ has $f_1 = f_2$, $\gamma_1=\gamma_2$ and $C_1 = C_2$. Now as $\frac{-\sqrt{2}\,i}{|H|}\tilde{H}=\xi-\bar{\xi}$, from the Frenet equations we obtain that the derivative of the function $\tilde{H}:\Sigma\rightarrow\r^6\,(\hbox{or}\,\r^6_2)$ is given by
\[
\tilde{H}_z=\frac{i|H|}{\sqrt{2}}(\xi_z-\bar{\xi}_z)=0.
\]
So $\tilde{H}=A$ for some vector $A\in\r^6\,(\hbox{or}\,\r^6_2)$ with $|A|=|H|>0$  and hence $0=\langle\Phi,\tilde{H}\rangle=\langle\Phi,A\rangle$ and $0=\langle\hat{\Phi},\tilde{H}\rangle=\langle\hat{\Phi},A\rangle$. Now if $\Phi=(\phi,\psi)$ and $A=(A_ 1,A_2)$, we finally get $\langle\phi,A_1\rangle=\langle\psi,A_2\rangle=0$. If $A_2=0$, we have that $A_1\not=0$ and so $\hbox{Jac}\,(\phi)=0$ and hence $C_1=-C_2$ which is not the case. So $A_2\not=0$ and $\langle\psi,A_2\rangle=0$, which implies that $\psi$ lies on a geodesic of $M^2(\epsilon)$. Hence, in this case, the immersion $\Phi$ factorizes through the totally geodesic hypersurface $M^2(\epsilon)\times\r$ as a CMC-surface.

Conversely given a PMC-immersion $\Phi:(\Sigma, g) \rightarrow M^2(\epsilon)\times M^2(\epsilon)$ such that $\Phi$ factorizes through the totally geodesic hypersurface $M^2(\epsilon)\times\r$ then from the proof of Lemma 1 we have that the data of $\Phi$ satisfy $f_1 = f_2$, $\gamma_1=\gamma_2$ and $C_1=C_2$. Hence, the corresponding data of $\Phi_1$ and $\Phi_2$ are the same and so they are congruent, i.e. $[\Phi_1]=[\Phi_2]$.
\end{proof}

The examples of product of curves of constant curvatures given in Exam\-ple~\ref{ex:producto-curvas} satisfy that $C_1=C_2=0$ and so in particular they are Lagrangian PMC-surfaces with respect to both complex structures. In the following result we classify (even locally) those PMC-surfaces of $M^2(\epsilon)\times M^2(\epsilon)$  which are Lagrangian with respect to some of the complex structures.
\begin{theorem}\label{tm:PMC-lagrangianas}
Let $\Phi:\Sigma\rightarrow M^2(\epsilon)\times M^2(\epsilon)$ be a  PMC-immersion of a surface $\Sigma$. If $\Phi$ is Lagrangian with respect to some of the K\"ahler structures $J_1$ or $J_2$, then $\Phi(\Sigma)$ is an open subset of some of the examples described in Example~\ref{ex:producto-curvas}.
\end{theorem}
\begin{remark}
This result is a generalization of Theorem 1 in \cite{CU}, where the authors proved the result when $\epsilon=1$, i.e. when the ambient space is $\s^2\times\s^2$ and the surface is compact.
\end{remark}

\begin{proof}
Taking the two-fold oriented covering of $\Sigma$ if necessary, we can assume that $\Sigma$ is orientable. Without lost of generality we suppose that $\Phi$ is a Lagrangian immersion with respect to $J_1$, i.e. $C_1=0$. Now, it is clear that $J_1H$ is a parallel tangent vector field to $\Sigma$ and hence $\Sigma$ is flat, i.e. $K=0$.

Now we are going to prove that the another K\"ahler function $C_2$ vanishes too and to do that we consider the holomorphic differential $\Theta_2$.

First, if $\Theta_2\equiv 0$, then from \eqref{eq:gradiente-C-general} and as $K=0$ one obtains that
\[
|\nabla C_2|^2=(1-C_2^2+4\epsilon|H|^2)\left(\frac{\epsilon(1-C_2^2)}{4}+|H|^2+\epsilon C_2^2\right).
\]
As $K=0$, \eqref{eq:laplaciano-C} becomes in
\[
\Delta C_2=-C_2(4|H|^2+\epsilon(1+C_2^2)).
\]
Hence the last two equations say that the function $C_2$ is isoparametric.
 Now we follow a standard reasoning. We work on the open set $U$ where $\nabla C_2\not=0$. We are going to prove that $U=\emptyset$ and so $C_2$ must be constant. As $K=0$, the Bochner formula says that
\[
\frac{1}{2}\Delta|\nabla C_2|^2=\langle \nabla C_2,\nabla(\Delta C_2)\rangle+\sum_{i=1}^2|\nabla_{e_i}\nabla C_2|^2,
\]
where $\{e_1,e_2\}$ is an orthonormal frame on $U$, and where we can take $e_1=\nabla C_2/|\nabla C_2|$.  Using the last two equations, i.e. that $C_2$ is isoparametric, it is not difficult to check that the Bochner formula becomes in
\[
0 = (4|H|^2 + \epsilon(1-C_2^2))\bigl( 3\epsilon(\epsilon + 4|H|^2)^2 - 18(\epsilon + 4|H|^2)C_2^2 - \epsilon C_2^4\bigr)
\]
So $C_2$ on $U$ satisfies the above non trivial polynomial and therefore $C_2$ must be constant on each connected component of $U$, which is impossible because $\nabla C_2 \neq 0$ on $U$.
We have proved that $U = \emptyset$. Therefore $C_2$ is constant. But $(C_2)_{z}=0$ implies that $(1-C_2^2)f_2=\frac{|H|}{\sqrt 2}\gamma_2^2$. From here and \eqref{eq:modulo-fj} one obtains that $C_2^2=\epsilon K=0$. So in this case our immersion $\Phi$ is also Lagrangian with respect to $J_2$.

Secondly if $\Theta_2\not=0$, then it has isolated zeroes. In this case from the integrability equations, the $1$-differential
\[
\Upsilon(z)=\gamma_1(z)\,(dz)
\]
is also holomorphic and without zeroes. Therefore $\Theta_2/\Upsilon^2$ is a holomorphic function. Let $p$ a point with $\Theta_2(p)\not=0$. Then in a connected neighborhood $U$ of $p$ we can normalize this holomorphic function as
\[
\Theta_2/\Upsilon^2=\lambda,\quad \lambda \in \r^*.
\]
Hence $|\Theta_2|^2=\lambda^2|\Upsilon|^4$. Now, from \eqref{eq:gradiente-C-general}, the integrability equations and the facts that $C_1=0$ and $K=0$ we get
\[
|\nabla C_2|^2=(1-C_2^2+4\epsilon|H|^2)\left(\frac{\epsilon(1-C_2^2)}{4}+|H|^2+\epsilon C_2^2\right)-\frac{\epsilon\lambda^2}{2}.
\]
As $K=0$, \eqref{eq:laplaciano-C} becomes in
\[
\Delta C_2=-C_2(4|H|^2+\epsilon(1+C_2^2)).
\]
In this second case the last two equations say that the function $C_2$ is also isoparametric on $U$.

Then, following a similar reasoning as in the first case, we obtain that $C_2=0$ on $U$. As this can be done at any point of $\Sigma$ except at the isolated zeroes of $\Theta_2$, we conclude, in this second case, that our immersion $\Phi$ is also Lagrangian with respect to $J_2$.

As a consequence, $\hbox{Jac}\,(\phi)=\hbox{Jac}\,(\psi)=0$ and the immersion $\Phi$ is the product of two curves. As the mean curvature is parallel we obtain the result.
\end{proof}

As we showed in the proof of Lemma~\ref{lm:diferencial-AR}, PMC-surfaces of $M^2(\epsilon)\times M^2(\epsilon)$ coming from CMC-surfaces of $M^2(\epsilon)\times\r$ have $C_1=C_2$ and in particular their extrinsic normal curvatures $\bar{K}^{\perp}=C_1^2-C_2^2$ vanish. Next theorem classifies the PMC-surfaces of $M^2(\epsilon)\times M^2(\epsilon)$ such that $\bar{K}^{\perp}=0$. Beside the above family, an interesting family of examples appears in the classification which we describe in the next result.
\begin{proposition}\label{prop:familia-ejemplos-K-normal-0}
Let $a,b,c$ be real numbers with $b>0$ and $h:I\subset\r\rightarrow\r$ a non-constant solution of the O.D.E.
\begin{equation}
(h')^2(x)=(a-h^2(x))\left((a-h^2(x))-\epsilon b(1+(h(x)-c)^2)\right),
\label{eq:ecuacion-h}
\end{equation}
satisfying $\epsilon(a-h^2(x))>0$, $\forall x\in I$. Then
$$\Phi=(\phi,\psi):I\times\r\rightarrow M^2(\epsilon)\times M^2(\epsilon)$$ where $\psi(x,y)=\psi(x)$ is the curve in $M^2(\epsilon)$ parametrizes by $|\psi'(x)|^2=b(1+(h(x)-c)^2)$ and with curvature $K_{\psi}(x)=-\frac{b\epsilon(a-h^2(x))}{|\psi'(x)|^3}$ and $\phi:I\times\r\rightarrow M^2(\epsilon)$ is given by
\begin{enumerate}[1.]
\item If $a > 0$,
\[
\phi(x,y) = \frac{1}{\sqrt{a}}\Bigl(\sqrt{\epsilon(a-h^2(x))}
 \cos(\sqrt{a}y),\sqrt{\epsilon(a-h^2(x))}  \sin(\sqrt{a}y), h(x)\Bigr),
 \]

\item If $a<0$ (which implies $\epsilon=-1$),
\[
\phi(x,y) = \frac{1}{\sqrt{-a}}\Bigl(h(x), \sqrt{h^2(x)-a} \sinh(\sqrt{-a}y), \sqrt{h^2(x)-a} \cosh(\sqrt{-a}y)\Bigr),
\]

\item If $a=0$ (which implies $\epsilon=-1$),
\[
\phi(x,y)=\frac{1}{2h(x)}\left((y^2-1)h^2(x)+1,2yh^2(x),(y^2+1)h^2(x)+1\right),
\]
\end{enumerate}
is a PMC-immersion.

All the examples described above satisfy $4|H|^2=b$, $C_1=C_2$ with $C_1^2=\frac{h'^2}{(a-h^2)^2}$, they are conformal immersions with  the induced metric given by $\epsilon(a-h(x)^2)(dx^2+dy^2)$ and the Hopf differentials given by
\[
\Theta_j=\frac{\epsilon b}{4}(a+1-c^2+2(-1)^{j}ic)(dz)^2,\quad j=1,2.
\]
\end{proposition}

\begin{remark}~
\label{rm:familia-ejemplos-K-normal-0}
\begin{enumerate}[1.]
    \item Following Proposition~\ref{prop:familia-ejemplos-K-normal-0}, the constant solutions of equation~\eqref{eq:ecuacion-h} satisfying $\epsilon(a-h^2)>0$ produce the PMC-surfaces of $M^2(\epsilon)\times M^2(\epsilon)$ with $C_1=C_2=0$, and so, from Theorem 1, they are the examples described in Example~\ref{ex:producto-curvas}.
    \item \label{item:isometrias:rm:familia-ejemplos-K-normal-0} All the previous examples are invariant under the $1$-parametric group of isometries $\{I(\theta) \times \mathrm{Id},\, \theta \in \r\}$ of $M^2(\epsilon) \times M^2(\epsilon)$, where $I(\theta):M^2(\epsilon)\rightarrow M^2(\epsilon)$ is the isometry given by:

        \begin{tabular}{ccc}
        $a > 0$ &$a < 0$ &$a = 0$\\[.1cm]
        $\begin{pmatrix}
            \cos \theta  & -\sin \theta  & 0 \\
            \sin \theta & \cos \theta & 0 \\
            0   &   0   &   1\\
        \end{pmatrix}$
        &$\begin{pmatrix}
            1 & 0 & 0 \\
            0 & \cosh \theta  & \sinh \theta \\
            0   &   \sinh \theta   &   \cosh\theta \\
        \end{pmatrix}$
        &$\begin{pmatrix}
            1 - \frac{\theta^2}{2} & \theta  & \frac{\theta^2}{2} \\
            -\theta   &   1  &   \theta  \\
            - \frac{\theta^2}{2} & \theta  & 1 + \frac{\theta^2}{2} \\
        \end{pmatrix}$
        \end{tabular}
\end{enumerate}
\end{remark}
\begin{proof}
First it is easy to check that, in the three cases,
\begin{gather*}
|\phi_x|^2 = \epsilon[(a-h^2) - \epsilon b(1+(h-c)^2)], \quad |\phi_y|^2 = \epsilon(a-h^2) \\
\langle \phi_x, \phi_y \rangle = 0, \quad \langle \phi_x, \phi_{xy} \rangle = 0, \quad \langle \phi_y, \phi_{xy} \rangle = -\epsilon h h'
\end{gather*}
 and hence $|\Phi_x|^2=|\Phi_y|^2=\epsilon(a-h^2)$ and $\langle \Phi_x, \Phi_y \rangle = 0$, which say that $\Phi$ is a conformal immersion. Now from a direct computation we have that
\begin{eqnarray*}
%\phi_{xx} = -\frac{\epsilon h'[h + \epsilon b (h-c)]}{|\phi_x|^2}\phi_x - \epsilon |\phi_x|^2\phi \\
%\phi_{yy}=\frac{\epsilon hh'}{|\phi_x|^2}\phi_x - \epsilon|\phi_y|^2\phi \\
%\phi_{xy}= \frac{-\epsilon hh'}{|\phi_y|^2}\phi_y \\
\phi_{xx} + \phi_{yy} = -\frac{bh'(h-c)}{|\phi_x|^2}\phi_x - \epsilon(|\phi_x|^2 + |\phi_y|^2)\phi.
\end{eqnarray*}
Also, the definition of the curve $\psi$ means that
\begin{eqnarray*}
\psi_{xx} + \psi_{yy} =\psi_{xx}= \frac{b h'(h-c)}{|\psi_x|^2}\psi_x - \frac{\epsilon b(a-h^2)}{|\psi_x|^2}J\psi_x - \epsilon |\psi_x|^2\psi.
\end{eqnarray*}
So therefore, as $\Phi$ is a conformal immersion, $H =  (\Phi_{xx} + \Phi_{yy})^T / 2\epsilon(a-h^2)$, where $(\,)^T$ denotes the tangential component to $M^2(\epsilon) \times M^2(\epsilon)$. Using the above formulae we get
\[
H = \frac{1}{2\epsilon(a-h^2)} \left( -\frac{b h'(h-c)}{|\phi_x|^2}\phi_x, \frac{b h'(h-c)}{|\psi_x|^2}\psi_x - \frac{\epsilon b(a-h^2)}{|\psi_x|^2}J\psi_x \right).
\]
From this equation the length of $H$ is $|H|^2 = b/4$ and after a long straightforward computation we obtain
\[
 \bar{ \nabla}_{\partial_x} H = -\frac{b(a-c h)}{2(a-h^2)} \Phi_x, \quad \bar{\nabla}_{\partial_y} H = \frac{bh(h - c)}{2(a-h^2)}\Phi_y,
\]
which proves that $H$ is parallel in the normal bundle.

Finally in order to compute the Hopf differentials we only need to know that
\[
\tilde{H} = \frac{1}{2\epsilon (a-h^2)}\left(\frac{bh'}{|\phi_x|^2}\phi_x, -\frac{bh'}{|\psi_x|^2}\psi_x - \frac{b\epsilon(a-h^2)(h-c)}{|\psi_x|^2}J\psi_x  \right)
\]
\end{proof}

Now, we are going to analyze the solutions of equation~\eqref{eq:ecuacion-h}. As the degree of the polynomial appearing in it is less than $5$, the solutions are elliptic functions which can be obtained knowing the roots of the polynomial. It is clear that every solution $h$ of equation~\eqref{eq:ecuacion-h} have not to satisfy the condition $\epsilon(a - h^2) > 0$, which is necessary to define a PMC-surface (without singularities).

If we denote by $p(t) = a - t^2$ and $q(t) = -(1 + \epsilon b)t^2 + 2\epsilon b c t - \epsilon b (1+c^2) + a$, the equation~\eqref{eq:ecuacion-h} becomes $(h')^2 = p(h)q(h)$. The condition $\epsilon(a - h^2) > 0$ means that $\epsilon p(h) > 0$ and so we obtain that $\epsilon q(h) \geq 0$ on certain interval of $\r$. This inequality of the two degree polynomial $q(h)$ gives us the restrictions
\begin{equation}
    \begin{split}
    (1+b)(a-b)\geq bc^2  &\qquad \text{if } \epsilon = +1 \\
    bc^2 \geq (b-1)(a+b) &\qquad \text{if } \epsilon = -1 \text{ and } 4|H|^2 = b > 1 \\
    c \neq 0 \text{ or } a \leq -1 &\qquad \text{if } \epsilon = -1 \text{ and } 4|H|^2 = b = 1
    \end{split}
    \label{eq:restricciones}
\end{equation}
about the parameters $a, b$ and $c$. On the other hand, it is possible to obtain all the solutions of equation~\eqref{eq:ecuacion-h} in terms of Jacobi elliptic functions (see \cite{BF}) and a deep analysis of them shows that the conditions appearing in \eqref{eq:restricciones} are also sufficient in order to the solutions of equation~\eqref{eq:ecuacion-h} satisfied $\epsilon(a-h^2)>0$. So
\begin{quote} {\it The solutions $h$ of the equation~\eqref{eq:ecuacion-h} verify $\epsilon(a - h^2)>0$  if and only if the parameters $a,b$ and $c$ of the equation satisfy the restrictions \eqref{eq:restricciones}.}
\end{quote}
The integration of equation~\eqref{eq:ecuacion-h} is not complicated but it is very long, because the roots of the polynomial appearing in the equation are of different nature depending on the values of the parameters $a,b$ and $c$ and hence the solutions of the equation are also of different nature. To illustrate the integration, we are going to integrate it in a particular case because the solution will produce a nice 1-parameter family of PMC-surfaces of $M^2(\epsilon)\times M^2(\epsilon)$.

\begin{example}\label{ex:a<-1-y-c=0}
We consider, in equation~\eqref{eq:ecuacion-h}, $\epsilon=-1$, $c=0$, $b = 1$ and from \eqref{eq:restricciones} $a \leq -1$. In this case the equation becomes in
\[
(h')^2(x)=(a+1)(a-h^2(x))
\]
and the solution $h$ with $h(0)=0$ is given by
\[
h(x) = \sqrt{-a}\sinh(\sqrt{-(1+a)}x).
\]
Hence $\epsilon (a-h^2) = -a\cosh^2(\sqrt{-(1+a)}x)$, and denoting $\lambda = \sqrt{-(1+a)}$ we get, from Proposition~\ref{prop:familia-ejemplos-K-normal-0}, that for all $\lambda \geq 0$, $\Phi_{\lambda}=(\phi,\psi):\r^2\rightarrow\h^2\times\h^2$ given by
\[
\phi(x,y)=\left(\sinh(\lambda x),\cosh(\lambda x)\sinh(\sqrt{1+\lambda^2}y)
,\cosh(\lambda x)\cosh(\sqrt{1+\lambda^2}y)\right)
\]
and $\psi(x,y)=\psi(x)$ the curve in $\h^2$ parametrized by $|\psi'|^2=1+(1+\lambda^2)\sinh^2(\lambda x)$ and with curvature $k_{\psi}=\frac{-\sqrt{1+\lambda^2}\cosh^2(\lambda x)}{|\psi'|^3}$ is a {\it PMC-conformal embedding of the complete surface $(\r^2,(1+\lambda^2)\cosh^2(\lambda x)(dx^2 + dy^2))$ with $4|H|^2=1$, $\Theta_1=\Theta_2=\frac{\lambda^2}{4}(dz)^2$}. The Gauss curvature of this metric is given by $K(x)=\frac{-\lambda^2}{\cosh^4(\lambda x)}$. When $\lambda = 0$, that is $a = -1$, $\Phi_0$ is the product of a geodesic and a horocycle, i.e. $\hat{P}_0$ in example~\ref{ex:producto-curvas}.
\end{example}

\begin{theorem}\label{tm:PMC-K-normal-nula}
Let $\Phi:\Sigma\rightarrow M^2(\epsilon)\times M^2(\epsilon)$ be a PMC-immersion of a surface $\Sigma$. Then the extrinsic normal curvature vanishes, $\bar{K}^{\perp}=0$, if and only if $\Phi$ is locally congruent to
\begin{enumerate}[1.]
\item a CMC-surface of $M^2(\epsilon)\times \r$,
\item one of the examples described in Example~\ref{ex:producto-curvas},
\item one of the examples described in Proposition $4$.
\end{enumerate}
\end{theorem}

\begin{remark}
Although $\bar{K}^{\perp}$ is well defined only for orientable surfaces, the equation $\bar{K}^{\perp}=0$, which means $C_1^2=C_2^2$, has sense even for non-orientable surfaces.
\end{remark}
\begin{proof}

First, it is clear that the examples given in 1) and 2) satisfy $\bar{K}^{\perp}=0$. Also, from Proposition~\ref{prop:familia-ejemplos-K-normal-0}, the examples given in 3) satisfy $\bar{K}^{\perp}=0$.

Suppose now that $\bar{K}^{\perp}=0$, i.e. $C_1^2=C_2^2$. Taking the two-fold oriented covering of $\Sigma$ if necessary, we can assume that $\Sigma$ is orientable.

From~\eqref{eq:laplaciano-C} we have that
\[
(\Delta+F)(C_1-C_2)=0,
\]
where $F=4|H|^2-2K+\epsilon(1+C_1^2)= 4|H|^2-2K+\epsilon(1+C_2^2)$. Now using the nodal Courant theorem, we obtain that either $C_1=C_2$ or $A=\{p\in\Sigma\,|\, C_1(p)=C_2(p)\}$ is a set of curves in $\Sigma$ and so as $C_1^2=C_2^2$ we have that $C_1+C_2=0$ on $\Sigma-A$ and hence on $\Sigma$. So we have two possibilities: $C_1=C_2$ or $C_1=-C_2$. It is clear that the surfaces with $C_1=-C_2$ can be obtained as the images of the surfaces with $C_1=C_2$ under the isometry $F:M^2(\epsilon)\times M^2(\epsilon)\rightarrow M^2(\epsilon)\times M^2(\epsilon)$ given by $F(p,q)=(q,p)$.

 Hence, we can assume that $C_1=C_2$. Then using the integrability equations~\eqref{eq:compatibilidad-PMC} we have that the $1$-differential
\[
\Omega(z)=(\gamma_2(z)-\gamma_1(z))dz
\]
is holomorphic. So either $\Omega\equiv 0$ or $\Omega$ has isolated zeroes. In the first case we have that $\gamma_1=\gamma_2$ and using that $(C_1)_z=(C_2)_z$ we obtain that $f_1=f_2$. Now, as in the proof of Theorem~\ref{tm:relacion-PMC-CMC}, it follows that $\Phi$ factorizes through a CMC-immersion of $M^2(\epsilon)\times\r$, and we obtain the case 1).

Now we study the case in which the holomorphic differential $\Omega$ is nonzero. Outside its zeroes we can normalize it as $\gamma_2-\gamma_1=2\sqrt{2}|H|$. As $C_1=C_2$ we have that $|\gamma_1|^2=|\gamma_2|^2$ and so $\Re\,(\gamma_1)=-\sqrt{2}|H|$. Hence
\[
\gamma_1=-\sqrt{2}|H|+ig,\quad \gamma_2=-\bar{\gamma}_1,
\]
for certain function $g:\Sigma\rightarrow\r$. Now, using the integrability equations~\eqref{eq:compatibilidad-PMC} we obtain that $g_{\bar{z}}=-\frac{e^{2u}}{\sqrt 2}|H|C_1$, which implies that $g(x, y) = g(x)$ and it satisfies $g'=-\sqrt{2}e^{2u}|H|C_1$, where $'$ stands for $\partial/\partial_x$. So from this equation and $e^{2u}(1-C_1^2)=2(2|H|^2+g^2)$ we deduce that $u$ and $C_1$ also satisfy $u(x,y) = u(x)$ and $C_1(x,y) = C_1(x)$.

Now, as $(\gamma_j)_z=(\gamma_j)_{\bar{z}}$, $j=1,2$, using again the integrability equations we have that $u'\gamma_j-2iC_jf_j=-\frac{ie^{2u}}{\sqrt 2}|H|C_j$, $j=1,2$. As $C_1=C_2$ and $\gamma_2=-\bar{\gamma}_1$, the above equations imply that $C_1(\bar{f_1}-f_2)=0$. Hence we have that either $C_1=C_2=0$ and $\Phi$ is the product of two curves of constant curvature and we prove 2) or $C_1^{-1}(\{0\})$ is a set of curves and so $\bar{f_1}=f_{2}$ on $\Sigma-C_1^{-1}(\{0\})$ and then on $\Sigma$.

Now we study this third case: $C_1=C_2$ non-null and $\bar{f}_1=f_2$. As $C_1$ is a function of $x$ and $(\Delta+F)(C_1)=0$, then the zeroes of $C_1$ are isolated. As $\bar{\gamma}_1^2=\gamma_2^2$, $\bar{f}_1=f_2$ and the Hopf differentials are holomorphic, we obtain that $\Theta_1=\mu(dz)^2$ and $\Theta_2=\bar{\mu}(dz)^2$ for certain complex number $\mu$. This says that
\[
2\sqrt{2}|H|f_1+\frac{\epsilon}{2}\gamma_1^2=\mu.
\]
In this situation, it is not difficult to see that the integrability conditions~\eqref{eq:compatibilidad-PMC} imply that
\begin{equation}
\begin{split}
u'&=C_1(\frac{\Im (\mu)}{2|H|^2}+\frac{\epsilon g}{\sqrt{2}|H|}), \\
e^{2u}(1-C_1^2)&=4|H|^2+2g^2, \quad g'=-\sqrt{2}e^{2u}|H|C_1.
\end{split}
\label{eq:derivada-u-g}
\end{equation}

We are going to integrate the Frenet equations. First of all, from~\eqref{eq:J1Phi-xi} and \eqref{eq:J2Phi-xi} we obtain that $J_1\Phi_z-J_2\Phi_z=2\Re\,(\gamma_1\xi)$ and  $J_1\Phi_z+J_2\Phi_z=2iC_1\Phi_z+2i\Im\,(\gamma_1\xi)$. So, taking into account the definitions of $J_j$, we get that $(0,J\psi_z)=\Re\,(\gamma_1\xi)$ and $(J\phi_z,0)=iC_1\Phi_z+i\Im\,(\gamma_1\xi)$. Hence
\begin{equation}
 J\psi_y=0,\, i.e.\quad \psi(x,y)=\psi(x),\quad \hbox{and}\quad J\phi_x=C_1\phi_y.
  \label{eq:Jpsi-Jphi}
\end{equation}

On the other hand, as $\bar{f_1}=f_2$, from the Frenet equations we have that
\[
\Phi_{zz}=u'\Phi_z+2\Re\,(f_1\xi)+\epsilon\frac{e^{2u}(1-C_1^2)}{4}\hat{\Phi},
\]
which implies, considering the imaginary part of this equation, that $\Phi_{xy}=u'\Phi_y$. This equation is irrelevant for the component $\psi$, but for the other component $\phi$, the equation $\phi_{xy}=u'\phi_y$ can be integrated to obtain that
\begin{equation}
\phi(x,y)=e^{u(x)}F(y)+G(x),
\label{eq:integracion-phi}
\end{equation}
for certain vectorial functions $F$ and $G$.

From the second equation of \eqref{eq:Jpsi-Jphi} and $\phi_{xy}=u'\phi_y$ it is easy to get that
\[
\phi_{yy}+\frac{u'}{C_1^2}\phi_x+\epsilon e^{2u}\phi=0.
\]
This equation joint with \eqref{eq:integracion-phi} say that the function $F$ satisfies the following O.D.E.
\[
F''(y)+(\frac{u'(x)^2}{C_1(x)^2}+\epsilon e^{2u(x)})F(y)+\tilde{G}(x)=0,
\]
for certain vectorial function $\tilde{G}$. From here, and taking derivatives with respect to $y$ and $x$ we get that $0=(\frac{u'(x)^2}{C_1(x)^2}+\epsilon e^{2u(x})'F'(y)$. But $e^{2u}=|\Phi_y|^2=|\phi_y|^2=e^{2u}|F'|^2$, which implies that $|F'|^2=1$. So from the above equations we finally obtain that
\begin{equation}
\frac{u'(x)^2}{C_1(x)^2}+\epsilon e^{2u(x)}=a\in\r,\quad \tilde{G}(x)=-\tilde{G}_0\in\r^3\,(\r^3_1),\quad\forall (x,y)\in\Sigma,
\label{eq:definicion-a}
\end{equation}
and so finally $F$ satisfies the following O.D.E.
\[
F''(y)+aF(y)-\tilde{G}{_0}=0.
\]
The solution of this equation is given by
\begin{eqnarray*}
F(y)=\cos(\sqrt{a}y)H_1(x)+\sin(\sqrt{a}y)H_2(x)+\frac{\tilde{G}_0}{a},\quad a>0,\\
F(y)=\cosh(\sqrt{-a}y)H_1(x)+\sinh(\sqrt{-a}y)H_2(x)+\frac{\tilde{G}_0}{a},\quad a<0,\\
F(y)=\frac{y^2}{2}\tilde{G}_0+yH_1(x)+H_2(x),\quad a=0.
\end{eqnarray*}
Using that $|F'|=1$ in these equations we obtain $H_1$ and  $H_2$ are constant, that is, $H_i(x) = H_i$, $i = 1, 2$ and that
\begin{eqnarray*}
|H_1|^2 = |H_2|^2 = 1/a,\quad \langle H_1, H_2\rangle = 0, &  a >0,\\
|H_1|^2 = -|H_2|^2 = 1/a,\quad \langle H_1, H_2\rangle = 0, &  a <0,\\
|H_1|^2 = 1, |\tilde{G}_0| = 0,\quad \langle H_1, \tilde{G}_0\rangle = 0, &  a =0
\end{eqnarray*}
Using this information in \eqref{eq:integracion-phi} we obtain that
\begin{eqnarray*}
\phi(x,y)=e^{u(x)}\cos(\sqrt{a}y)H_1+e^{u(x)}\sin(\sqrt{a}y)H_2+\hat{G}(x),\quad  a>0,\\
\phi(x,y)=e^{u(x)}\cosh(\sqrt{-a}y)H_1+e^{u(x)}\sinh(\sqrt{-a}y)H_2+\hat{G}(x),\quad  a<0,\\
\phi(x,y)=e^{u(x)}\frac{y^2}{2}\tilde{G}_0+e^{u(x)}yH_1+\hat{G}(x),\quad a=0.
\end{eqnarray*}
 As $\langle \phi, \phi_y \rangle = 0$, we deduce from the above equations that $\langle \hat{G}(x), H_j\rangle = 0$, $j = 1, 2$, $\langle \hat{G}(x), \tilde{G}_0\rangle = -e^{u(x)}$. Now, up to an isometry in $\r^3$ or $\r^3_1$ we can choose $H_1 = (1/\sqrt{a},0,0)$, $H_2 = (0,1/\sqrt{a},0)$ and $\hat{G} = h(x)(0, 0, 1/\sqrt{a})$ when $a > 0$, $H_1 = (0,0,1/\sqrt{-a})$, $H_2 = (0,1/\sqrt{-a},0)$ and $\hat{G} = h(x)(1/\sqrt{-a}, 0, 0)$ when $a < 0$, and $H_1 = (0,1,0)$, $\tilde{G}_0 = (1,0,1)$ and $\hat{G} = (\frac{1-h^2(x)}{2h(x)}, 0, \frac{1-h^2(x)}{2h(x)}+e^{u(x)})$ when $a = 0$, for certain function $h$. Therefore, the above equations become in
\begin{align*}
\phi(x,y) &= \frac{1}{\sqrt{a}}\Bigl(e^{u(x)} \cos(\sqrt{a}y), e^{u(x)} \sin(\sqrt{a}y), h(x)\Bigr), & a>0 \\
\phi(x,y) &= \frac{1}{\sqrt{-a}}\Bigl(h(x), e^{u(x)} \sinh(\sqrt{-a}y), e^{u(x)} \cosh(\sqrt{-a}y)\Bigr), & a<0 \\
\phi(x,y) &= \left( \frac{e^{u(x)}}{2}y^2 +\frac{1-h^2(x)}{2h(x)}, e^{u(x)}y, \frac{e^{u(x)}}{2}y^2  + \frac{1-h^2(x)}{2h(x)} + e^{u(x)}\right), & a = 0,
\end{align*}
where $h(x)^2 + \epsilon e^{2u(x)}= a$.

To study the curve $\psi(x)$, from \eqref{eq:Jpsi-Jphi} and as $|\phi_y|^2 = e^{2u}$ we have that
\[
|\psi_x|^2=e^{2u}-|\phi_x|^2 = e^{2u} - |J\phi_x|^2 = e^{2u} - C_1^2|\phi_y|^2 = e^{2u}(1-C_1^2)=4|H|^2+2g^2.
\]
Moreover, taking into account the Frenet equation for $\Phi_{z\bar{z}}$,  \eqref{eq:J1Phi-xi} and \eqref{eq:J2Phi-xi} we have:
\[
\begin{split}
\langle \psi_{xx}, J\psi_x\rangle &= \langle 4\Phi_{z\bar{z}}, (0, J\psi_x)\rangle = 2e^{2u}\langle H, (0, J\psi_x) \rangle = \\
&= \sqrt{2}e^{2u}|H| \langle \xi + \bar{\xi}, (0, J\psi_x) \rangle = \frac{e^{2u}|H|}{\sqrt{2}|\gamma_1|^2} (\gamma_1 + \bar{\gamma}_1) \langle J\psi_x, J\psi_x\rangle = \\
&= \frac{1}{|\gamma_1|^2}(-2|H|^2e^{2u})|\psi_x|^2 = -\frac{4|H|^2}{1-C_1^2}|\psi_x|^2=-4|H|^2e^{2u}
\end{split}
\]
and so $k_\psi(x) = -4|H|^2e^{2u}/|\psi_x|^3.$

To check that these examples are the given in Proposition~\ref{prop:familia-ejemplos-K-normal-0}, we only need to get the O.D.E. that $h$ satisfied. From \eqref{eq:definicion-a} and as $h^2+\epsilon e^{2u}=a$, we have that $h=-\frac{\epsilon u'}{C_1}$ and so \eqref{eq:derivada-u-g} implies that $h=-\frac{\epsilon\Im(\mu)}{2|H|^2}-\frac{ g}{\sqrt 2|H|}$. From \eqref{eq:derivada-u-g} again we get that $h'= e^{2u}C_1$ and then using one more time \eqref{eq:derivada-u-g} we get
\begin{eqnarray*}
(h')^2=C_1^2e^{4u}=e^{2u}(e^{2u}-4|H|^2-2g^2)\\=(a-h^2)(a-h^2-\epsilon 4|H|^2(1+(h+\epsilon\frac{\Im(\mu)}{2|H|^2})^2).
\end{eqnarray*}
Now, if we define $b=4|H|^2$ and $c=-\epsilon\frac{\Im(\mu)}{2|H|^2}$, we obtain that $h$ satisfies the equation of Proposition~\ref{prop:familia-ejemplos-K-normal-0} and the curve $\psi(x)$ satisfies
\[
|\psi'|^2=|\psi_x|^2=b(1+(h-c)^2)\quad\hbox{and}\quad k_{\psi}=-\frac{\epsilon b(a-h^2)}{|\psi'|^3}.
\]
So, in this third case our surface is one of the examples described in Proposition~\ref{prop:familia-ejemplos-K-normal-0} and we have finished the proof.
\end{proof}

\begin{theorem}\label{tm:PMC-diferenciales-nulas}
Let $\Phi:\Sigma\rightarrow M^2(\epsilon)\times M^2(\epsilon)$ be a PMC-immersion of an orientable surface $\Sigma$. The Hopf differentials vanish, i.e. $\Theta_1=\Theta_2=0$, if and only if one of the three following possibilities happens:
\begin{enumerate}[1.]
\item $\Phi(\Sigma)$ lies in $M^2(\epsilon)\times\r$ as a CMC-surface with vanishing Abresh-Rosenberg differential,
\item $\epsilon=-1$, $4|H|^2=1$ and locally $\Phi$ is the product of two hypercycles $\alpha$ and $\beta$ of $\h^2$ with curvatures $k_{\alpha}^2+k_{\beta}^2=1$,
\item  $\epsilon=-1$, $4|H|^2<1$ and locally $\Phi$ is $\Phi_0=(\phi_0,\psi_0):
(-\pi/2,\pi/2)\times\r\rightarrow\h^2\times\h^2,$ where
\[
\phi_0(x,y)=\frac{1}{\cos x}\left( \sin x,\sinh \frac{y}{\sqrt{1-2|H|^2}},\cosh \frac{y}{\sqrt{1-4|H|^2}}\right),
\]
and $\psi_0$ is the curve in $\h^2$ given by $|\psi_0'(x)|=\frac{2|H|}{\sqrt{1-4|H|^2}\cos x}$ and with curvature $k_0(x)=-\frac{\cos x}{2|H|}$.
\end{enumerate}
\end{theorem}
\begin{remark}
$\Phi_0$ is a conformal embedding and the induced metric \\
$\frac{1}{(1-4|H|^2)\cos^2 x}\,(dx^2+dy^2)$ is complete and with constant curvature $4|H|^2-1$. Moreover $C_1^2=C_2^2=1-4|H|^2$.

In \cite{L} Leite described, for $|H|^2<1/4$, a CMC-isometric embedding of the hyperbolic plane with constant curvature $4|H|^2-1$ in $\h^2\times\r$ and with vanishing Abresh-Rosenberg differential. So we have two non-congruent PMC-isometric embeddings of the hyperbolic plane with constant curvature $4|H|^2-1$ into $\h^2\times\h^2$.
\end{remark}

\begin{proof}
 Suppose that $\Theta_1=\Theta_2=0$. Then we have that $16|H|^2|f_j|^2=\frac{|\gamma_j|^4}{2}$, $j=1,2$, which means that
\begin{equation}
|H|^2+\epsilon C_j^2-K=\frac{(1-C_j^2)^2}{16|H|^2},\quad j=1,2.
\label{eq:condicion-Theta-0}
 \end{equation}
 From this equation we easily get that
\[
(C_1^2-C_2^2)\left( 16\epsilon|H|^2+(1-C_1^2)+(1-C_2^2)\right)=0.
\]
If $\epsilon=1$, from the above equation we obtain that $C_1^2=C_2^2$. If $\epsilon=-1$, on the open set $O=\{p\in\Sigma\,|\,C_1^2(p)\not=C_2^2(p)\}$, we have that
\begin{equation}
C_1^2+C_2^2=2(1-8|H|^2).
\label{eq:suma-cuadrados-C_j}
\end{equation}
But on $O$, $C_1\nabla C_1=-C_2\nabla C_2$, and then using \eqref{eq:gradiente-C-general}, \eqref{eq:condicion-Theta-0} and the integrability equations~\eqref{eq:compatibilidad-PMC} we obtain that
\[
C_1^2(1-C_1^2)(1-C_1^2-4|H|^2)^2=C_2^2(1-C_2^2)(1-C_2^2-4|H|^2)^2.
\]
As on $O$, $8|H|^2\not=1$, using \eqref{eq:suma-cuadrados-C_j} we obtain that $C_j$, $j=1,2$ satisfy a nontrivial polynomial of degree $6$, which implies that $C_j$, $j=1,2$ are constant on each connected component of $O$. But using again \eqref{eq:gradiente-C-general} we get that either $C_j=0$ or $1-C_j^2=4|H|^2$ on each connected component of $O$. This contradicts \eqref{eq:suma-cuadrados-C_j}, and so $O=\emptyset$ and hence in this case ($\epsilon=-1$) $C_1^2=C_2^2$ too.

Therefore $\bar{K}^{\perp}=0$ and from Theorem~\ref{tm:PMC-K-normal-nula} we have three possibilities. In the first case, $\Phi(\Sigma)$ lies as a CMC-surface in $M^2(\epsilon)\times\r$ and, from Lemma~\ref{lm:diferencial-AR}, it has vanishing Abresh-Rosenberg differential.

In the second case, $C_j=0,\,j=1,2$, and $\Phi$ is locally one of the examples of Example~\ref{ex:producto-curvas}. As $\Theta_1=\Theta_2=0$, Lemma 1,2) says that $\epsilon=-1$ and $4|H|^2=1$. This fact only happens for the product of two suitable hypercycles or for the product of a horocycle and a geodesic, but the latter is a particular case of 1). So we have proved 2).

Finally, in the third case we have a PMC-surface described in Proposition~\ref{prop:familia-ejemplos-K-normal-0} with $\Theta_j=0,\,j=1,2$. But then $a=-1$ and $c=0$, which implies that $\epsilon=-1$. In this case, equation~\eqref{eq:ecuacion-h} becomes in
\[
(h')^2=(1-4|H|^2)(1+h^2)^2,
\]
which implies that $4|H|^2\leq 1$.

If $4|H|^2=1$, then $h$ is constant and $\Phi$ is congruent to the product of a geodesic and a horocycle which, up to a congruence, is included in case 1).

If $4|H|^2<1$, the solution of the above equation is given by $h(x)=\tan(\sqrt{1-4|H|^2}x)$ with $-\frac{\pi}{2}<\sqrt{1-4|H|^2}x< \frac{\pi}{2}$. Now, reparametrizing the immersion by $(x, y) \rightarrow (\sqrt{1-4|H|^2}x, y)$, the PMC-immersion associated to $h$ in Proposition 4 is $\Phi_0$. Hence we get 3).

The converse is clear.
\end{proof}

\begin{corollary}\label{cor:PMC-spheres}
Let $\Phi:\Sigma\rightarrow M^2(\epsilon)\times M^2(\epsilon)$ be a PMC-immersion of a sphere $\Sigma$. Then, up to congruences, $\Phi$ is a CMC-sphere in $M^2(\epsilon)\times\r$.
\end{corollary}
The examples described in Theorem~\ref{tm:PMC-diferenciales-nulas}.3) and the examples obtained by Leite in \cite{L} can be characterized in the following way.
\begin{corollary}
Let $\Phi:\Sigma\rightarrow M^2(\epsilon)\times M^2(\epsilon)$ be a PMC-immersion of an orientable surface $\Sigma$. Then the extrinsic and normal extrinsic curvatures $\bar{K}$ and $\bar{K}^{\perp}$ are constant if and only if one of the two following possibilities happens
\begin{enumerate}[1.]
\item $\bar{K}=\bar{K}^{\perp}=0$ and $\Phi$ is locally congruent to some of the examples described in Example~\ref{ex:producto-curvas},
\item $\bar{K}=4|H|^2-1$, $\bar{K}^{\perp}=0$, and
 $\Phi$ is locally congruent either the example given in Theorem~\ref{tm:PMC-diferenciales-nulas}.3)  or the example described by Leite in \cite{L}.
\end{enumerate}
\end{corollary}
\begin{proof}
First $\bar{K}$ and $\bar{K}^{\perp}$ are constant if and only if $C_j$, $j=1,2$ are constant. Also, the examples of Example~\ref{ex:producto-curvas} satisfy $C_j=0$, $j=1,2$ and the examples of Theorem~\ref{tm:PMC-diferenciales-nulas}.3) and the given by Leite satisfy $C_j^2=1-4|H|^2$ and $\epsilon=-1$.

On the other hand, if $C_j$, $j=1,2$ are constant, from the integrability equations~\eqref{eq:compatibilidad-PMC} we have that $(1-C_j^2)f_j=\frac{|H|}{\sqrt{2}}\gamma_j^2$, $j=1,2$, and hence, computing their lengths and using~\eqref{eq:modulo-fj}, $C_1^2=C_2^2=\epsilon K$. So either $C_j=0$, $j=1,2$ and we obtain 1) or $C_1=C_2$ is a non-null constant. In the latter, from \eqref{eq:laplaciano-C} we obtain that $\epsilon=-1$ and $C_j^2=1-4|H|^2$, $j=1,2$. Using all this information we can check that $\Theta_j=0$, $j=1,2$. The result is now a consequence of Theorem~\ref{tm:PMC-diferenciales-nulas} and \cite{L}.
\end{proof}

\section{Examples of CMC-surfaces in $M^2(\epsilon)\times\r$}

Following Theorem~\ref{tm:relacion-PMC-CMC}, the examples of PMC-surfaces of $M^2(\epsilon)\times M^2(\epsilon)$ des\-cri\-bed in Proposition~\ref{prop:familia-ejemplos-K-normal-0} have associated pairs of CMC-surfaces of $M^2(\epsilon)\times\r$. As these PMC-surfaces do not factorize through CMC-surfaces of $M^2(\epsilon)\times\r$, the pairs of CMC-surfaces are not congruent.

Let $\Phi:I\times\r\rightarrow M^2(\epsilon)\times M^2(\epsilon)$ be a CMC-surface associated to a solution $h$ of ~\eqref{eq:ecuacion-h} in Proposition~\ref{prop:familia-ejemplos-K-normal-0}. Following the proof of Theorem~\ref{tm:PMC-K-normal-nula}, the Frenet data associated to this immersion are given by
\begin{gather*}
u(z)=\log\sqrt{\epsilon(a-h^2(x))},\, C_1(z)=C_1(\bar{z})=C_2(z),\, f_2(z)=f_2(\bar{z})=\bar{f}_1(z),\\  \gamma_2(z)=-\bar{\gamma}_1(z),\quad \gamma_1(z)=\sqrt{2}|H|(1+i(h(x)-c)).
\end{gather*}
Hence the Frenet data associated to the pair $(\Phi_1,\Phi_2)$ of CMC-surfaces of $M^2(\epsilon)\times\r$ (see proof of Theorem~\ref{tm:relacion-PMC-CMC}) are given by
\begin{gather*}
u(z),\, \nu_1(z)=\nu_1(\bar{z})=\nu_2(z),\, p_2(z)=p_2(\bar{z})=\bar{p}_1(z),\\
\eta_1(z)=-2|H|(y+\int_{x_0}^x(h(t)-c)dt),\, \eta_2(z)=-2|H|(-y+\int_{x_0}^x(h(t)-c)dt).
\end{gather*}
 As the map $G:I\times\r\rightarrow I\times\r$ given by $G(z)=\bar{z}$ is an isometry of the induced metric $g=\epsilon(a-h^2(x))(\mathrm{d}x^2+\mathrm{d}y^2)$, it is easy to check that $\Phi_1\circ G$ is a CMC-surface with the same Frenet data than $\Phi_2$, and so $\Phi_1\circ G$ and $\Phi_2$ are congruent immersions, i.e. $\Phi_1$ and $\Phi_2$ are weakly congruent. So really there is only a CMC-immersion associated to each PMC-immersion of Proposition~\ref{prop:familia-ejemplos-K-normal-0}. In this case we can also integrate the Frenet equations of these immersions, obtaining the following family of examples.

\begin{proposition}\label{prop:familia-ejemplos-CMC}
Let $a,b,c$ be real numbers with $b>0$ and $h:I\subset\r\rightarrow\r$ a non-constant solution of the O.D.E.~\eqref{eq:ecuacion-h} satisfying $\epsilon(a-h^2(x))>b$, $\forall x\in I$. Then $\Psi = (\psi, \eta): I \times \r \rightarrow M^2(\epsilon)\times \r$ where $\eta(x,y) = \sqrt{b}\left( y+\int_{x_0}^x \bigl(h(t) - c\bigr)\mathrm{d}t \right)$ and $\psi:I\times\r\rightarrow M^2(\epsilon)$ is given by
\begin{enumerate}
  \item If $E = a - \epsilon b > 0$
  \[
    \psi(x,y) = \frac{1}{\sqrt{E}}\left( \sqrt{\epsilon(E - h(x)^2)} \cos (\sqrt{E}f),
    \sqrt{\epsilon(E - h(x)^2)}\sin (\sqrt{E} f),
    h(x) \right)
  \]
  \item If $E < 0$ (which implies $\epsilon=-1$)
  \[
    \psi(x,y) = \frac{1}{\sqrt{-E}}\left(h(x), \sqrt{h(x)^2 - E}\sinh (\sqrt{-E}f),  \sqrt{h(x)^2 - E}\cosh (\sqrt{-E}f) \right)
  \]
  \item If $E = 0$ (which implies $\epsilon = -1$)
  \[
    \psi(x, y) = h(x)\left(f^2 - \frac{1}{4} + \frac{1}{h(x)^2}, f,
    f^2 + \frac{1}{4} + \frac{1}{h(x)^2} \right)
  \]
\end{enumerate}
is a CMC-immersion, where in the three cases
\[
f(x, y) = y + \int_{x_0}^x \frac{b(c - h(t))}{\epsilon(E - h^2(t))} \mathrm{d} t
\]

All the examples described above satisfy $4|H|^2 = b$, they are conformal immersions with the induced metric given by $\epsilon(a-h^2(x))(\mathrm{d}x^2 + \mathrm{d}y^2)$ and their Abresh-Rosenberg differentials given by
\[
  \Theta_{AR} = \frac{\epsilon b}{8} \left(a + 1 - c^2 - 2ic\right) (\mathrm{d}z)^2
\]
\end{proposition}
\begin{remark}~
  \begin{enumerate}
    \item Because $\epsilon(a - h^2(x)) > b > 0$ the parameters $a$, $b$ and $c$ have to satisfy~\eqref{eq:restricciones}. Reciprocally if $a$, $b$ and $c$ satisfy~\eqref{eq:restricciones} then there exists a non-constant solution $h:I \rightarrow \r$ with $\epsilon(a -h^2(x)) > 0$. Now from~\eqref{eq:ecuacion-h} $\epsilon(a - h^2(x)) - b \geq 0$ so, as $h$ is non constant, there exists $I' \subseteq I$ such that $\epsilon(a - h^2(x)) - b > 0$. Therefore $\epsilon(a - h^2(x)) > b$ (for a suitable interval $I'$) if and only if $a$, $b$ and $c$ satisfy~\eqref{eq:restricciones}.
    \item All these examples are invariant under the $1$-parametric group of isometries $\{I(\theta) \times \tau_\theta,\, \theta \in \r\}$ of $M^2(\epsilon) \times \r$, where $\tau_\theta: \r \rightarrow \r$ is $\tau_\theta(t) = t + \theta \sqrt{b}$ and $I(\theta): M^2(\epsilon) \rightarrow M^2(\epsilon)$ is the isometry given in remark~\ref{rm:familia-ejemplos-K-normal-0}.\ref{item:isometrias:rm:familia-ejemplos-K-normal-0}.
  \end{enumerate}
\end{remark}

\begin{proof}
First it is easy to check that, in the three cases,
\begin{gather*}
  |\psi_x|^2 = \epsilon (a - h^2(x)) - b(h(x)-c)^2, \quad |\psi_y|^2 = \epsilon(a - h^2(x)) - b, \\ \langle \psi_x, \psi_y\rangle = -b(h(x) - c),\quad  \langle \psi_x, \psi_{xy} \rangle = 0, \quad \langle \psi_y, \psi_{xy}\rangle = -\epsilon h(x)h'(x) .
\end{gather*}
So taking into account the definition of $\eta$ and that $\Psi = (\psi, \eta)$ we get
\[
  |\Psi_x|^2 = |\Psi_y|^2 = \epsilon(a - h^2(x)), \quad \langle \Psi_x, \Psi_y\rangle = 0,
\]
that is, $\Psi$ is a conformal immersion with conformal factor $\epsilon(a - h^2(x))$. Then its mean curvature vector field is given by $H = (\Psi_{xx} + \Psi_{yy})^T/2\epsilon(a-h^2(x))$, where $()^T$ denotes the tangential component to $M^2(\epsilon) \times \r$. So, by a direct computation we get
\begin{align*}
  H(x, y) = \frac{\sqrt{b}}{2} \left(\frac{\sqrt{b}}{h'(x)}\bigl((c-h(x)) \psi_x - \psi_y\bigr), \frac{h'(x)}{\epsilon(a-h^2(x))} \right)
\end{align*}
 From this we have that $\Psi$ is a CMC-immersion with $|H|^2 = b/4$ and it is straightforward to check that the associated Abresh-Rosenberg differential is:
\[
  2\Theta_{AR}(z) =  \frac{\epsilon b}{4}\bigl(a +1 - c^2 -2ic\bigr) (\mathrm{d}z)^2
\]
\end{proof}
From Theorem~\ref{tm:relacion-PMC-CMC} and Theorem~\ref{tm:PMC-K-normal-nula} we can obtain the following rigidity result for CMC-surfaces of $M^2(\epsilon)\times\r$.
\begin{corollary}\label{cor:rigidez}
Let $\Phi_1,\Phi_2:(\Sigma,g)\rightarrow M^2(\epsilon)\times\r$ be  two non-congruent CMC-isometric immersions of a simply-connected surface $\Sigma$ with the same mean curvatures $H_1=H_2$ and the same extrinsic curvatures $\bar{K}_1=\bar{K}_2$. Then $\Phi_1$ and $\Phi_2$ are weakly congruent, i.e. there exists an isometry $G$ of $(\Sigma,g)$ such that $\Phi_1\circ G$ and $\Phi_2$ are congruent, and $\Phi_1$ is either one of the examples of Proposition~\ref{prop:familia-ejemplos-CMC} or $\Phi_1$ is a cylinder over a curve of constant curvature of $M^2(\epsilon)$.
\end{corollary}
\begin{proof}
From Theorem~\ref{tm:relacion-PMC-CMC}, let $\Phi:(\Sigma,g)\rightarrow M^2(\epsilon)\times M^2(\epsilon)$ be the PMC-isometric immersion associated to the pair $(\Phi_1,\Phi_2)$. Then its extrinsic normal curvature is given by $\bar{K}^{\perp}= \epsilon \frac{C_1^2 - C_2^2}{2} = \frac{\bar{K}_1-\bar{K}_2}{2}=0$. So, as $\Phi_1$ and $\Phi_2$ are not congruent, Theorem~\ref{tm:PMC-K-normal-nula} says that $\Phi$ is either one of the examples of
Proposition~\ref{prop:familia-ejemplos-K-normal-0} or the product of two curves of constant curvature. At the beginning of this section it was proved that, in the first case, $\Phi_1$ and $\Phi_2$ are weakly congruent and $\Phi_1$ is one of the examples of Proposition~\ref{prop:familia-ejemplos-CMC}. In the second case, if $\Phi$ is given by $\Phi(x,y)=(\alpha(x),\beta(y))$ with $\alpha$ and $\beta$ curves in $M^2(\epsilon)$ with constant curvature $k_{\alpha}$ and $k_{\beta}$ and $|\alpha'|=|\dot\beta|=1$, following Theorem~\ref{tm:relacion-PMC-CMC}, the Frenet data of $\Phi_1$ and $\Phi_2$ are given by $u = 0$, $H_1 = H_2$, $\nu_1 = \nu_2 = 0$, $p_2 = \bar{p}_1$ and $\eta_2(x, y) = \eta_1(-x, y) = -(k_\beta x + k_\alpha y)/\sqrt{k_\alpha^2 + k_\beta^2}$. As $G(x, y) = (-x ,y)$ is an isometry of the induced metric $g = dx^2 + dy^2$ then $\Phi_1 \circ G$ is a CMC-immersion with the same Frenet data than $\Phi_2$, and so $\Phi_1$ and $\Phi_2$ are weakly congruent. Now if $\Phi_1 = (\phi_1, \eta_1)$ it is not difficult to check that the image of $\phi_1$ is a curve in $M^2(\epsilon)$ with constant curvature $\sqrt{k_\alpha^2 + k_\beta^2}$ and the corollary is proved.
\end{proof}
Among the examples described in Proposition~\ref{prop:familia-ejemplos-CMC} there are some of them of particular interest that we are going to describe.

\begin{example}\label{ex:a<-1CMC}
We consider the $1$-parameter family of CMC-immersions in Proposition~\ref{prop:familia-ejemplos-CMC} associated to the PMC-immersions given in example~\ref{ex:a<-1-y-c=0} ($\epsilon=-1$, $c=0$, $b=1$). Following the notation, for each $\lambda> 0$,
 $\Psi_{\lambda} = (\psi_{\lambda}, \eta_{\lambda}): \r^2 \rightarrow \h^2\times \r$, where:
\[
  \begin{split}
  \psi_{\lambda}(x, y) &= \frac{\sqrt{1+\lambda^2}}{\lambda}\left( \sinh x, \cosh x \sinh y + \frac{\cosh y}{\sqrt{1+\lambda^2}},\cosh x \cosh y + \frac{\sinh y}{\sqrt{1+\lambda^2}}\right) \\
  \eta_{\lambda}(x, y) &= \frac{1}{\lambda}(y+\sqrt{1+\lambda^2} \cosh x) , \\
  \end{split}
\]
  is a {\it CMC-conformal isometric embedding of the complete surface \\ $\bigl(\r^2,\frac{1+\lambda^2}{\lambda^2}\cosh^2x(dx^2 + dy^2)\bigr)$ in $\h^2\times\r$ with $H=1/2$}. Its Abresh-Rosenberg differential is given by $(dz)^2/8$.
\end{example}

\begin{example}
  We consider the CMC-immersion in proposition~\ref{prop:familia-ejemplos-CMC} associated to the example $\Phi_0$ of Theorem~\ref{tm:PMC-diferenciales-nulas}. Following the notation, for each real number $0<H<1/2$, $\Psi_0=(\psi_0,\eta_0):]-\pi/2,\pi/2[\times\r\rightarrow\h^2\times\r$ given by
   \[
    \begin{split}
    \psi_0(x,y) &= \frac{1}{\sqrt{1-4H^2}} \left(\tan x,\frac{\sinh y}{\cos x} + 2H^2e^{-y}\cos x,\frac{\cosh y}{\cos x} -2H^2e^{-y}\cos x\right) \\
    \eta_0(x,y) &= \frac{2H}{\sqrt{1-4H^2}}\left(\, y - \log \cos x\right),
    \end{split}
  \]
is a {\it CMC-conformal isometric embedding with mean curvature H of the hyperbolic plane $\left(]-\pi/2,\pi/2[\times\r,1/(1-4H^2)\cos^2 x\right)$ with curvature $4H^2-1$ in $\h^2\times\r$}. Its Abresh-Rosenberg differential vanishes and it is a conformal reparametrization of the Leite example \cite{L}.
\end{example}

\begin{example}
Now we are going to obtain examples of CMC-tori in $\s^2\times\s^1$. To do that, first we need to get periodic solutions of the O.D.E.~\eqref{eq:ecuacion-h}. We consider $\epsilon=1$, $c=0$ and from \eqref{eq:restricciones} $a>b$, and them equation~\eqref{eq:ecuacion-h} becomes in
  \[
  (h')^2(x)=(a-h^2(x))\left(a - b -(1+b)h^2(x)\right)=q(h).
  \]
   As the roots of the polynomial q are $\pm\sqrt{a}, \pm\sqrt{(a-b)/(1+b)}$, formula 219.00 in \cite{BF} says that the solution $h:\r \rightarrow \r$ of the above equation with $h(0)=0$ is given by
  \[
  h(x) = \sqrt{\frac{a-b}{1+b}} \sn(\sqrt{a(1+b)}x)
  \]
  where $\sn$ is the sine amplitude Jacobi function with modulus $\kappa^2 = (a-b)/a(1+b)$. These solutions are periodic with period $4K(\kappa)/\sqrt{a(1+b)}$ where $K(\kappa)$ is the complete elliptic integral of the first kind.

  In this case
  \[
  \epsilon(a-h^2(x))=a(1-\kappa^2\sn^2(\sqrt{a(1+b)}x))=a\dn^2(\sqrt{a(1+b)}x)>0,\quad \forall x\in\r,
  \]
  where $\dn $ is the delta amplitude Jacobi function. Furthermore, $\epsilon(a-h^2(x)) > b$ because the minimum for the function $\dn$ is $\sqrt{1-\kappa^2}$ and it is easy to see that $a (1-\kappa^2) > b$ if and only if $a > b$.

 Now the function $f$ appearing in Proposition~\ref{prop:familia-ejemplos-CMC} is given by
 \[
 f(x, y) = y + \frac{1}{\sqrt{a-b}}\arctan \left( \frac{\cn(\sqrt{a(1+b)}x)}{\sqrt{a}\dn(\sqrt{a(1+b)}x)}\right),
 \]
 where $\cn$ is the cosine amplitude Jacobi function. Then, up to the reparametrization $(x,y)\mapsto \frac{1}{\sqrt{a(1+b)}}(x,y)$, the associated  CMC-immersion $\Phi_{a,b} = (\phi_{a,b}, \eta_{a,b}):\r^2 \rightarrow \s^2 \times \r$ is given by:
  \[
    \begin{split}
       \phi_{a,b}(x, y) &= \left(\frac{\sqrt{a}\dn x\cos (\kappa y)-\cn x\sin (\kappa y)}{\sqrt{1+a}},\frac{\sqrt{a}\dn x\sin (\kappa y)+\cn x\cos (\kappa y)}{\sqrt{1+a}}, \frac{\sn x}{\sqrt{1+b}} \right)\\
        \eta_{a,b}(x, y) &= \frac{\sqrt{b}}{\sqrt{1+b}} \log ( \dn x - \kappa \cn x) + \frac{\sqrt{b}}{\sqrt{a(1+b)}} y.
    \end{split}
  \]
  We consider the local isometry $t\in\r\mapsto \frac{\sqrt b}{\sqrt{a-b}}e^{i\frac{\sqrt{a-b}}{\sqrt b}t}\in\s^1(\frac{\sqrt b}{\sqrt{a-b}})$ and the CMC-immersion
  \[
  \hat{\Phi}_{a,b}=(\phi_{a,b},\hat{\eta}_{a,b}):\r^2\rightarrow \s^2\times\s^1(\frac{\sqrt b}{\sqrt{a-b}}),
  \]
  where $\hat{\eta}_{a,b}(x,y)=\frac{\sqrt b}{\sqrt{a-b}}e^{i\frac{\sqrt{a-b}}{\sqrt{1+b}}\log (\dn x - \kappa \cn x)}e^{i \kappa y}.$

  It is clear that $\hat{\Phi}$ is invariant under the group  $G_{a,b}$ of transformations of $\r^2$ generated by
  \[
  (x,y)\mapsto (x+4K(\kappa),y),\quad (x,y)\mapsto \left(x,y+\frac{2\pi}{\kappa}\right).
  \]
  If $T_{a,b}=\r^2/G_{a,b}$ is the associated torus and $P:\r^2\rightarrow T_{a,b}$ the projection, then the induced immersion
  \[
  \tilde{\Phi}_{a,b}:T_{a,b}\rightarrow \s^2\times\s^1(\frac{\sqrt b}{\sqrt{a-b}}),\quad P(x,y)\mapsto \hat{\Phi}_{a,b}(x,y),
  \]
defines a CMC-conformal immersion of the torus $T_{a,b}$ into $\s^2\times\s^1(\frac{\sqrt b}{\sqrt{a-b}}).$

We are going to see that $\tilde{\Phi}_{a,b}$ is an embedding. In fact, if $\tilde{\Phi}_{a,b}(P(x,y))=\tilde{\Phi}_{a,b}(P(\hat{x},\hat{y}))$, with $x,\hat{x}\in [0,4K(\kappa)[,\, y,\hat{y}\in[0,2\pi/\kappa[$, then we have that $\sn x=\sn \hat{x}$ and then either $x=\hat{x}$ or $x,\hat{x}\in[0,2K(\kappa)]$ and $x+\hat{x}=2K(\kappa)$ or  $x,\hat{x}\in[2K(\kappa),4K(\kappa)]$ and $x+\hat{x}=6K(\kappa)$. In the first case, looking at the immersion we obtain that $y=\hat{y}$. In the other two cases, $\cn \hat{x}=-\cn x$ and $\dn \hat{x}=\dn{x}$. So, looking again at the immersion we easily get that
\[
\cos (\kappa\hat{y}-\kappa y)=\frac{a\dn^2x-\cn^2x}{a\dn^2x+\cn^2x},\, \cos (\kappa\hat{y}-\kappa y)=\cos\log\left(\left(\frac{\dn x-k\cn x}{\dn x+k\cn x}\right)^{\frac{\sqrt{a-b}}{\sqrt{1+b}}}\right).
\]
From these equations we obtain that $x=K(\kappa)$ or $x=3K(\kappa)$, which implies that $\hat{x}=x$. Again, $y=\hat{y}$, and so our immersion is an embedding.

We can summarize the above reasoning in the following result.
\begin{proposition}\label{prop:toros}
For each pair of real numbers $0<b<a$, the immersion $\hat{\Phi}_{a,b}:T_{a,b}\rightarrow\s^2\times\s^1(\frac{\sqrt b}{\sqrt{a-b}})$ described above is a CMC-conformal embedding of the rectangular torus $T_{a,b}$ with mean curvature $H=\sqrt{b}/2$. Its Abresh-Rosenberg differential is $\Theta_{AR} = (b(1+a)/8a(1+b)) (dz)^2$.
\end{proposition}
\begin{remark}
It is clear that $\tau:T_{a,b}\rightarrow T_{a,b}$ defined by $$\tau(P(x,y))=P\left(-x,y+\frac{\pi}{k}\right)$$ is an isometry of $T_{a,b}$ with $\tau^2=\mathrm{Id}$. Because $\tilde{\Phi}_{a,b}\bigl(\tau P(x, y)\bigr) = -\tilde{\Phi}_{a,b}\bigl( P(x, y) \bigr)$ for all $(x, y) \in \r^2$, $\tilde{\Phi}_{a,b}$ induces a CMC-embedding of the Klein bottle $B_{a, b} = T_{a,b}/\langle\tau\rangle$ in $\mathbb{RP}^2 \times \mathbb{RP}^1(\sqrt{b}/\sqrt{a-b})$, where $\mathbb{RP}^2$ denotes the real projective plane with constant curvature $1$ and $\mathbb{RP}^1(\sqrt{b}/\sqrt{a-b})$ denotes the real projective line with constant curvature $\sqrt{a-b}/\sqrt{b}$.
\end{remark}
\end{example}

\section{Compact PMC-surfaces}
In this section we are going to prove some properties of {\it compact} PMC-surfaces of $M^2(\epsilon)\times M^2(\epsilon)$. Let $\Phi:\Sigma\rightarrow M^2(\epsilon)\times M^2(\epsilon)$ be an PMC-immersion of an orientable surface $\Sigma$. We define two vector fields $X_j,\,j=1,2$, tangent to $\Sigma$ as the tangential components of $J_j\tilde{H}$,
\[
J_1\tilde{H}=X_1+C_1H, \quad J_2\tilde{H}=X_2-C_2H.
\]
In particular we have that $|X_j|^2=|H|^2(1-C_j^2),\, j=1,2$.
Differentiating these equations and taking tangential components we obtain that
\[
\nabla_vX_1=C_1\,A_{H}v-C_1J^{\Sigma}A_{\tilde{H}}v,\quad\nabla_vX_2=-C_2\,A_{H}v-C_2J^{\Sigma}A_{\tilde{H}}v
\]
for any tangent vector $v$, where $J^{\Sigma}$ is the complex structure of the Riemann surface $\Sigma$.
From here we obtain that the divergence of $X_j$ and the differential of the $1$-forms $\alpha^{X_j}(v)=\langle X_j,v\rangle$ are given by
\begin{equation}
\hbox{div}\,X_j=(-1)^{j+1}2C_j|H|^2,\quad d\alpha^{X_j}=0,\,j=1,2.
\label{eq:divergencia-X}
\end{equation}
Now, using the above properties and that $|X_j|^2=|H|^2(1-C_j^2)$, the Bochner formula becomes in
\[
\frac{1}{2}\Delta(1-C_j^2)=K(1-C_j^2)+(-1)^{j+1}2\langle \nabla C_j,X_j\rangle+\frac{|\nabla X_j|^2}{|H|^2},\quad j=1,2.
\]
Using now the expression of the covariant derivative of $X_j$ we finally get that
\begin{equation}
\frac{1}{2}\Delta(1-C_j^2)=K(1-C_j^2)+(-1)^{j+1}2\langle \nabla C_j,X_j\rangle+2C_j^2(\epsilon C_j^2+2|H|^2-K),
\label{eq:gradiente-1-C^2}
\end{equation}
On the other hand, as $\Delta(1-C_j^2)=-2C_j\Delta C_j-2|\nabla C_j|^2,\, j=1,2$, from equation \eqref{eq:gradiente-1-C^2} and \eqref{eq:laplaciano-C}, we obtain that
\begin{equation}
|\nabla C_j|^2=(1-C_j^2)(\epsilon C_j^2-K)+(-1)^{j}2\langle \nabla C_j,X_j\rangle,\quad j=1,2.
\label{eq:norma-gradiente-C-respecto-X}
\end{equation}

All these formulae have some consequences when the surface is compact.
\begin{proposition}\label{prop:PMC-compacta}
Let $\Phi=(\phi,\psi):\Sigma\rightarrow M^2(\epsilon)\times M^2(\epsilon)$ be an immersion of a {\bf compact} orientable surface with parallel mean curvature vector. Then
\begin{enumerate}[1.]
\item $\int_{\Sigma}C_j\,dA=0,\quad j=1,2.$
\item If $\epsilon=1$, then the degrees of $\phi$ and $\psi$ are zero.
\item If $K\geq 0$, then either $\Phi(\Sigma)$ is a CMC-sphere of $\s^2\times\r$ with $4|H|^2\geq 1$ or $\Phi(\Sigma)$ is a torus of Example~\ref{ex:producto-curvas}.
\item If $\epsilon=1$ then $K$ cannot be negative. If $\epsilon=-1$, then $K$ cannot be less than $-1$.
\item If some of the holomorphic differentials $\Theta_j$ vanishes, then also vanishes the other and so $\Phi(\Sigma)$ is a CMC-sphere of $M^2(\epsilon)\times\r$.
\end{enumerate}
\end{proposition}

\begin{proof}
Integrating the first equation of \eqref{eq:divergencia-X} we prove (1). If $\epsilon=1$, then $\phi,\psi:\Sigma\rightarrow \s^2$ are maps such that (see section 3)
\[
\phi^*\omega=\frac{C_1+C_2}{2}\omega_{\Sigma},\quad \psi^*\omega=\frac{C_1-C_2}{2}\omega_{\Sigma},
\]
which proves (2) making use of (1).

If $K\geq 0$, then either $\Sigma$ is a sphere and Theorem~\ref{tm:PMC-diferenciales-nulas} proves that it is a CMC-sphere of $M^2(\epsilon)\times\r$, or $\Sigma$ is a flat torus. In the first case only the CMC-spheres of $\s^2\times\r$ with $4|H|^2\geq 1$ have non-negative Gauss curvature. In the second case, from \eqref{eq:divergencia-X} we have that
\begin{equation}
0=\int_{\Sigma}\hbox{div}\,(C_jX_j)\,dA=\int_{\Sigma}\langle \nabla C_j,X_j\rangle\,dA+(-1)^{j+1}2|H|^2\int_{\Sigma}C_j^2\,dA,
\label{eq:ecucion-integral-divergencia-C-X}
\end{equation}
that, joint with the integration of equation~\eqref{eq:gradiente-1-C^2}, gives us
\[
0=\int_{\Sigma}(K(1-3C_j^2)+2\epsilon C_j^4)\,dA.
\]
As $\Sigma$ is flat, we obtain that $C_j=0$, $j=1,2$, and so $\Phi(\Sigma)$ is a torus of Example~\ref{ex:producto-curvas}.

Now we prove (4). From \eqref{eq:norma-gradiente-C-respecto-X}, $p$ is a critical point of $C_j$ if and only if either $C_j^2(p)=1$ or $K(p)=\epsilon C_j^2(p)$. If $\epsilon=1$ and $K<0$ or $\epsilon=-1$ and $K<-1$ the second possibility cannot happen and hence all the critical points satisfy $C^2(p)=1$. Taking into account Proposition~\ref{prop:C-tiene-puntos-aislados}, the function $C_j$ is a Morse function with only maximum and minimum as critical points. Therefore the surface must be a sphere, but the Gauss-Bonnet theorem gives a contradiction. This proves (4).

Finally if some of the holomorphic differentials vanishes, i.e. $\Theta_1=0$, then from \eqref{eq:gradiente-C-general}, \eqref{eq:norma-gradiente-C-respecto-X} and \eqref{eq:ecucion-integral-divergencia-C-X} we obtain that
\[
16|H|^2\int_{\Sigma}K\,dA=\int_{\Sigma}(4|H|^2+\epsilon(1-C_1^2))^2\,dA.
\]
In particular $\int_{\Sigma}K\,dA\geq 0$ and again either $\Sigma$ is a sphere and so $\Theta_2=0$ or $\Sigma $ is a torus in Example~\ref{ex:producto-curvas} with $\Theta_1=0$, which is impossible looking at Lemma~\ref{lm:diferencial-AR}.
\end{proof}


\begin{thebibliography}{1}
  \bibliographystyle{alpha}

\bibitem{AR}
U.~Abresch and H.~Rosenberg.
\newblock A Hopf differential for
constant mean curvature surfaces in $\s^2\times\r$ and
$\h^2\times\r$.
\newblock {\em Acta Math.} {\bf 193} (2004) 141--174.

\bibitem{B}
D. Ben\^oit.
\newblock  Isometric immersions into 3-dimensional homogeneous manifolds.
\newblock{\em Comment. Math. Helv.} {\bf 82} (2007) 87--131.

\bibitem{BF}
P.F. Byrd and M.D. Friedman.
\newblock {\em Handbook of Elliptic Integral for Engineers and Scientists.}
\newblock Springer--Verlag, New York (1971).

\bibitem{CU}
I. Castro and F. Urbano.
\newblock Minimal Lagrangian surfaces in $\s^2\times\s^2$.
\newblock {\em Comm. Anal. and Geom.} {\bf 15} (2007) 217--248.


\bibitem{C}
B.Y. Chen.
 \newblock On the surface with parallel mean curvature vector.
 \newblock {\em Indiana Univ. Math. J.} {\bf 22} (1973) 655--666.

\bibitem{EGT}
J.H. Eschenburg, I.V. Guadalupe and R.A. Tribuzy.
\newblock The fundamental equations of minimal surfaces in $\c\p^2$.
\newblock {\em Math. Ann.} {\bf 270} (1985) 571--598.

\bibitem{FM}
I. Fern{\'a}ndez and P. Mira.
\newblock A characterization of constant mean curvature surfaces in homogeneous 3-manifolds.
\newblock {\em Differential Geom. Appl.} {\bf 25} (2007) 281--289.


\bibitem{F}
D. Ferus.
\newblock The torsion form of submanifolds in $E^N$.
\newblock {\em Math. Ann.} {\bf 193} (1971) 114--120.

\bibitem{KZ}
K. Kenmotzu and D. Zhou.
\newblock The classification of the surfaces with parallel mean curvature vector in two-dimensional complex space forms.
\newblock {\em Amer. J. Math.} {\bf 122} (2000) 295--317.

\bibitem{L}
M.L. Leite.
\newblock An elementary proof of the Abresh-Rosenberg theorem on constant mean curvature immersed surfaces in $\s^2\times\r$ and $\h^2\times\r$.
\newblock{Quart.J. Math.} {\bf 58} (2007) 479-487.

\bibitem{LV}
J.H.S. de Lira and A. Vitorio.
\newblock Surfaces with constant mean curvature in Riemannian products.
\newblock {\em Preprint}

\bibitem{O}
T. Ogata,
\newblock Surfaces with parallel mean curvature in $\p^2(\c)$.
\newblock{\em Kodai Math. J.} {\bf 18} (1995) 397-407.

\bibitem{Y}
S.T. Yau.
\newblock  Submanifolds with constant mean curvature I.
\newblock{\em Amer. J. Math.} {\bf 96} (1974) 346--366.

\end{thebibliography}
\end{document}